\documentclass[10pt,a4paper]{article}
\usepackage{amssymb,amsmath,amsthm}
\usepackage[latin1]{inputenc}
\usepackage[english]{babel} 
\usepackage{algorithm,algorithmic}
\usepackage{color}
\newcommand{\ep}{\varepsilon}

\newcommand{\ssq}{\subseteq}
\newcommand{\AB}[1]{ow}

\newtheorem{theorem}{Theorem}[section]
\newtheorem{proposition}[theorem]{Proposition}
\newtheorem{question}[theorem]{Question}

\newtheorem{conjecture}[theorem]{Conjecture}
\newtheorem*{definition}{Definition}

\newtheorem{corollary}[theorem]{Corollary}
\newtheorem{lemma}[theorem]{Lemma}
\newtheorem*{remark}{Remark}

\begin{document}

\title{Subgraphs of weakly quasi-random oriented graphs}

\author{\quad\\Omid Amini\,$^\ast$, \ Simon Griffiths\,$^\dagger$, \ and \
  Florian Huc\,$^\ddagger$\\[3mm]
  $^\ast$\, CNRS -- DMA, \emph{\'Ecole Normale Sup\'erieure, Paris,
    France}\\[1mm]
  $^\dagger$\,\emph{IMPA, Est. Dona Castorina 110, Jardim Bot\^anico, Rio de Janeiro, Brazil}\\[1mm]
  $^\ddagger$\,\emph{CUI, Universit\'e de Gen\`eve, Geneva, Switzerland}}

\maketitle

{\renewcommand{\thefootnote}{\relax} \footnotetext{
E-mail\,: \texttt{oamini-at-dma.ens.fr},
    \texttt{sgriff-at-cs.mcgill.ca}, \texttt{florian.huc-at-unige.ch}}}

\begin{abstract} It is an intriguing question to see what kind of information on the structure of an oriented graph $D$ one can obtain if $D$ does not contain a fixed oriented graph $H$ as a subgraph. 
The related question in the unoriented case has been an active area of research,
 and is relatively well-understood in the theory of quasi-random graphs and extremal combinatorics.  

 In this paper, we consider the simplest cases of such a general question for oriented graphs,
 and provide some results on the global behavior of the orientation of  $D$.  
For the case that $H$ is an oriented four-cycle we prove: in every $H$-free oriented graph $D$,
 there is a pair $A,B\ssq V(D)$ such that $e(A,B)\ge e(D)^{2}/32|D|^{2}$ and $e(B,A)\le e(A,B)/2$.  
We give a random construction which shows that this bound on $e(A,B)$ is best possible 
(up to the constant).  
In addition, we prove a similar result for the case $H$ is an oriented six-cycle, 
and a more precise result in the case $D$ is dense and $H$ is arbitrary.  
We also consider the related extremal question in which no condition is put on the oriented graph $D$, 
and provide an answer that is best possible up to a multiplicative constant. 
Finally, we raise a number of related questions and conjectures.\end{abstract}

\section{Introduction}

Many results in graph theory show that local properties of a graph have global consequences.  
An example, which is particularly relevant to us, comes from the theory of quasi-randomness \cite{CGW}, 
and shows that if a large dense graph $G$ does not contain a fixed graph $H$ as a subgraph, 
then it must contain a large (linear-size) subgraph whose density is significantly 
different from that of the original graph.  
Thus, the restriction that $G$ does not contain $H$ as a subgraph forces $G$ 
to exhibit irregular behaviour on a global scale.  
In another direction, it has been conjectured by Erd\H os and Hajnal \cite{EH89} 
that for each graph $H$, there is a constant $\alpha >0$ such that every graph $G$ 
that does not contain an induced copy of $H$, must have a clique or independent set of size 
at least $n^{\alpha}$.  
A partial result, proving the existence of a clique or independent set of size $\exp(c \sqrt{\log{n}})$, has been given \cite{EH89} (see also \cite{EH77}).

We prove results for oriented graphs which also show that large scale irregularities can 
be deduced from local properties.  
In this sense our results are quasi-randomness type results for oriented graphs.  
That said, they differ significantly in style from the work on quasi-randomness 
by Thomason \cite{Thom1, Thom2}; Chung, Graham and Wilson \cite{CGW}; and
 Chung and Graham \cite{CG}, which is of a more precise form and applies to dense graphs only.  
In fact, \cite{CG} is concerned with tournaments (orientations of complete graphs), 
and does indeed include a proof that excluding a fixed tournament 
as a subgraph (or even bounding the number of copies away from the expected value) 
does lead to consequences on a global scale 
(for example, vertex subsets $X,Y$ with $|e(X,Y)-e(Y,X)|=\Omega(n^{2})$).  
They actually prove much more -- that eleven separate quasi-randomness conditions are equivalent.  
A more direct analogue of these results for general (dense) oriented graphs 
has been considered by the second author \cite{Gri09}.

Throughout the article we write $D$ for an oriented graph, $n$ for the number of vertices of $D$, 
and $e(D)$ for the number of arcs (oriented edges) in $D$.  For a pair of (non-necessarily disjoint) subsets $A,B\ssq V$,
 we write $e(A,B)$ for the number of arcs from $A$ to $B$, i.e., the size of the set 
$E(A,B)=\{\vec{xy}\in E(D):x\in A, y\in B\}$.  We say that a subgraph $E(A,B)$ is \emph{biased} 
if $e(B,A)\le e(A,B)/2$.  (The choice of the fraction $\frac{1}{2}$ is rather arbitrary, 
it could be replaced by any $\eta\in (0,1)$ without any significant effect on our results.)  
We write $bias(D)$ for the size of the largest biased subgraph in $D$:
\begin{equation*} bias(D):=\max\Big\{e(A,B) :\, A,B\ssq V\, \text{such that}\, e(B,A)\le \frac{e(A,B)}{2}\Big\}\, .\end{equation*}

The parameter $bias(D)$ measures irregularities in the orientation of $D$.  
If $bias(D)$ is small, then one might say that $D$ has a random-like orientation. 
 For example, if $D$ is obtained by orienting edges at random, 
then $bias(D)=O(n)$ with high probability (see Lemma \ref{random}).  In general, however,
 $bias(D)$ may be as large as $e(D)$, the number of arcs of $D$.  We now state our main result.

\begin{theorem}\label{fourcycle} There exists a constant $\ep>0$ such that every oriented graph $D$ with $bias(D)<\ep e(D)^{2}/n^{2}$ contains an oriented four-cycle.\end{theorem}

\begin{remark} \rm This result is best possible, up to the choice of $\ep$.  In Section \ref{bpsec} we construct a family of oriented graphs $D$ which have $bias(D)<Ke(D)^{2}/n^{2}$ but do not contain oriented four-cycles, where $K$ is a fixed constant.
\end{remark}

\begin{theorem}\label{sixcycle} There exists a constant $\ep>0$ such that every oriented graph $D$ with $bias(D)<\ep e(D)^{2}/n^{2}$ contains an oriented six-cycle. \end{theorem}

We have no reason to believe that this result is best possible.  In fact we conjecture a stronger result.

\begin{conjecture}\label{sixcycleconj} There exists a constant $\ep>0$ such that every oriented graph $D$ with $bias(D)<\ep e(D)^{3/2}/n$ contains an oriented six-cycle. \end{conjecture}

We also have a more general conjecture concerning oriented cycles of all even lengths.  

\begin{conjecture}\label{evencycleconj} For each $k\ge 2$, there exists a constant $\ep>0$ such that every oriented graph $D$ with $bias(D)<\ep e(D)^{k/(k-1)}/n^{2/(k-1)}$ contains an oriented cycle of length $2k$. \end{conjecture}

\begin{remark}\rm The analogous question for oriented cycles of odd length is not so interesting.  
Indeed, a random orientation $D$ of the complete bipartite graph with parts of order 
$\lfloor n/2\rfloor$ and $\lceil n/2\rceil$ has $bias(D)$ very small (of order $n$) while $D$ 
does not contain any oriented cycle of odd length.  Looking for odd length cycles may become 
interesting in the case the underlying graph is not arbitrary, cf. Concluding Remarks (Section \ref{ConcRem}).\end{remark}

In addition to guaranteeing the presence of an oriented cycle of prescribed length, the above conditions can also be used to guarantee many such cycles.

\begin{theorem}\label{fourcycles}
There exist constants $c,\ep >0$ such that every oriented graph $D$ with $bias(D)< \ep e(D)^{2}/n^{2}$
contains at least $ce(D)^{4}/n^{4}$ oriented four-cycles.
\end{theorem}

\begin{theorem}\label{sixcycles} There exist constants $c,\ep >0$ such that every oriented graph $D$ with $bias(D)<\ep e(D)^{2}/n^{2}$ contains at least $ce(D)^{6}/n^{6}$ oriented six-cycles.\end{theorem}

\begin{remark} \rm The number of cycles obtained is (up to the
constant) best possible.  A random orientation of the
random graph $G(n,e)$ (recall that $G(n,e)$ is a graph selected uniformly at
random from the class of all graphs on $n$ vertices with $e$ edges) will
generally obey the $bias$ condition (see Lemma \ref{random}), and have $\Theta(e^{4}/n^{4})$ oriented four-cycles and $\Theta(e^{6}/n^{6})$ oriented six-cycles.\end{remark}

The above results are relatively general in the sense that they apply 
in the case $D$ is sparse as well as the case $D$ is dense, 
while they are restrictive in the sense that they only look for oriented four-cycles 
and oriented six-cycles.  The next theorem does the opposite -- restricting to the case $D$ is dense, 
it counts copies of any oriented graph $H$.  

In what follows we count homomorphic copies of subgraphs.  
Let $H$ be an oriented graph on $k$ vertices.  
We write $hom(H,D)$ for the number of homomorphic copies of $H$ in $D$, i.e., 
the number of functions $\phi:V(H)\rightarrow V(D)$ such that $\vec{\phi(x)\phi(y)}\in E(D)$ for every arc $\vec{xy}\in E(H)$.  Likewise, we write $hom(\bar{H},D)$ for the number of homomorphic copies of the unoriented graph $\bar{H}$ in $D$, 
i.e., the number of functions $\phi:V(H)\rightarrow V(D)$ such that for every edge $\{x,y\}\in E(\bar{H})$ either $\vec{\phi(x)\phi(y)}$ or $\vec{\phi(y)\phi(x)}$ is an arc of $D$.

\begin{theorem}\label{dense} Let $H$ be an oriented graph on $k$ vertices and let $D$ be an oriented graph with $bias(D)<\ep n^{2}$.  Then \begin{equation*} hom(H,D)\ge \frac{hom(\bar{H},D)}{3^{e(H)}}-\frac{\ep}{2}n^{k}\, .\end{equation*}\end{theorem}

Since one may easily bound the number of degenerate homomorphic copies of $H$, one easily deduces the following corollary.

\begin{corollary}\label{densecor} Let $H$ be an oriented graph on $k$ vertices and let $D$ be an oriented graph with $bias(D)<\ep n^{2}$.  If $hom(\bar{H},D)\ge 3^{e(H)}\ep n^{k}$ and $n\ge 4\ep^{-1}$, then $D$ contains $H$ as a subgraph.\end{corollary}

Since the main results of the article may be viewed as 
lower bounds on $bias(D)$ for $H$-free oriented graphs $D$ (for given choices of $H$), 
it is natural to also consider the related extremal question: 
\begin{quote}What is the minimum value of $bias(D)$ over oriented graphs on $n$ (non-isolated) vertices?\end{quote}  
It will be seen from the results we obtain that this is indeed the most natural form of 
the extremal question.  
This question will be considered in detail in Section \ref{regran}.  
The main results establish that the minimum value of $bias(D)$ over oriented graphs on $n$ 
non-isolated vertices is $\Theta(n/\log{n})$.  We also prove that the answer is larger, $\Theta(n)$, 
in the case that the oriented graph is regular, or close to regular.  
In addition, we provide an algorithmic proof of the lower bound in the case $D$ is regular.

\medskip
The layout of the article is as follows.  
In Section \ref{fourcyclessec} we prove Theorem \ref{fourcycle} and Theorem \ref{fourcycles}.  
In Section \ref{sixcyclessec} we prove Theorem \ref{sixcycle} and Theorem \ref{sixcycles}.  
In Section \ref{bpsec} we describe a family of oriented graphs which show that Theorem \ref{fourcycle} 
is best possible (up to the choice of $\ep$).  
In Section \ref{densesec} we consider the case of dense oriented graphs, and provide a proof of
 Theorem \ref{dense}.  In Section \ref{regran} we discuss extremal questions concerning 
the parameter $bias(D)$.  Finally, we give in Section \ref{ConcRem} a number of concluding remarks.  
These remarks include new questions and conjectures, together with further discussion of the 
conjectures stated above. We should note here that it is possible to obtain another proof of 
Corollary~\ref{densecor} (with some different (worse) constants) using 
Szemer\'edi's regularity lemma. The details of this proof are included in the Appendix of the arxiv version of this article (arXiv:0911.3969).

\section{Oriented Four-Cycles}\label{fourcyclessec}

We begin by defining some notation.  Let $x$ be a vertex of an oriented graph $D$.  We shall use the following notation: \begin{align*} 
\Gamma^{+}(x) & =  \{y\in V:\vec{xy}\in E(D)\}, 
\\ \Gamma^{-}(x) & = \{y\in V:\vec{yx}\in E(D)\},
\\ \Gamma^{++}(x) & = \{y\in V: \exists z\in V\, \vec{xz},\vec{zy}\in E(D)\},
\\ \Gamma^{--}(x) & =  \{y\in V: \exists z\in V\, \vec{yz},\vec{zx}\in E(D)\}.
\end{align*}
We let $d^{+}(x)=|\Gamma^{+}(x)|$ and $d^{-}(x)=|\Gamma^{-}(x)|$.  We also define the following notation for joint-degrees, we let $d^{++}(x,u)=|\Gamma^{+}(x)\cap \Gamma^{+}(u)|$ and $d^{+-}(x,u)=|\Gamma^{+}(x)\cap \Gamma^{-}(u)|$.

We now prove a useful lemma.  Recall that throughout $D$ denotes an oriented graph on $n$ vertices with $e$ arcs.

\begin{lemma}\label{twopaths} Let $D$ be an oriented graph with $bias(D)\le e(D)/2$. Then $D$ contains at least $e(D)^{2}/8n$ paths of length two.\end{lemma}

\begin{proof} The number of paths of length two in $D$ is \begin{equation*} \sum_{y\in V} d^{+}(y)d^{-}(y)\, .\end{equation*} Denote by $Z$ the set of vertices $y$ for which $d^{-}(y)>2d^{+}(y)$.  By summing over vertices $y\in Z$ one finds that $e(V,Z)> 2e(Z,V)$, and so, by the definition of $bias(D)$, we have that $e(V,Z)\le bias(D)\le e/2$.  Therefore $e(V,Y)\ge e(D)/2$, where $Y=V\setminus Z$.  We also have that $d^{+}(y)\ge d^{-}(y)/2$ for all $y\in Y$, so that \begin{equation*} \sum_{y\in V}d^{+}(y)d^{-}(y)\ge \frac{1}{2}\sum_{y\in Y}d^{-}(y)^{2}\ge \frac{1}{2n}\Big (\sum_{y\in Y}d^{-}(y)\Big)^{2}\, ,\end{equation*} where the final inequality follows from an application of the Cauchy-Schwarz inequality.  The proof is now complete, as $\sum_{y\in Y}d^{-}(y)=e(V,Y)\ge e(D)/2$.\end{proof}

We now turn to the proof of Theorem \ref{fourcycle}.  We shall prove this theorem with $\ep=1/32$.  Thus, throughout the proof we may assume that $D$ has $bias(D)< e(D)^{2}/32n^{2}$.  Note also the following useful formulation of the $bias(D)<\ep e(D)^{2}/n^{2}$ property:
\begin{equation*} \begin{array}{lr} e(B,A)\ge e(A,B)/2 \quad \text{whenever}\quad e(A,B)\ge \ep e(D)^{2}/ n^{2}\, .& \qquad (\star)\end{array}\end{equation*}
Also we introduce a final piece of notation.  For each vertex $x$, we write $e_{x}$ for the number of paths of length two in $D$ which start at $x$.  Equivalently, $e_{x}=e(\Gamma^{+}(x),\Gamma^{++}(x))$.

\begin{proof}[Proof of Theorem \ref{fourcycle}] Let $D$ be an oriented graph with $bias(D)< e(D)^{2}/32n^{2}$.  By Lemma \ref{twopaths}, we know that there are at least $e(D)^{2}/8n$ paths of length two in $D$.  Denote by $W$ the set of vertices $x$ with $e_{x}\ge e(D)^{2}/16n^{2}$.  Since at most $e(D)^{2}/16n$ paths of length two start at vertices outside of $W$, we have that $\sum_{x\in W}e_{x}\ge e(D)^{2}/16n$.

For each vertex $x\in W$, we have $e(\Gamma^{+}(x),\Gamma^{++}(x))=e_{x}\ge e(D)^{2}/16n^{2}$ and so, by $(\star)$, we have that $e(\Gamma^{++}(x),\Gamma^{+}(x))\ge e_{x}/2$.  Equivalently \begin{equation*}\sum_{u\in \Gamma^{++}(x)}d^{++}(x,u)\ge \frac{e_{x}}{2}\, .\end{equation*}
Summing over $x\in W$, we obtain \begin{equation*}\sum_{x\in W}\sum_{u\in \Gamma^{++}(x)}d^{++}(x,u)\ge \sum_{x\in W}\frac{e_{x}}{2}\ge \frac{e(D)^{2}}{32n}\, .\end{equation*}  We now consider a change in the order of summation.
\begin{equation*}\sum_{u\in V}\sum_{x\in \Gamma^{--}(u)}d^{++}(x,u)\ge \sum_{x\in W}\sum_{u\in \Gamma^{++}(x)}d^{++}(x,u)\ge \frac{e(D)^{2}}{32n}\, .\end{equation*} 
In particular this implies that for some $u\in V$ one has $\sum_{x\in \Gamma^{--}(u)}d^{++}(x,u)\ge e(D)^{2}/32n^{2}$.  Equivalently, $e(\Gamma^{--}(u),\Gamma^{+}(u))\ge e(D)^{2}/32n^{2}$.  A final application of $(\star)$ gives that $$e(\Gamma^{+}(u),\Gamma^{--}(u))\ge e(D)^{2}/64n^{2}>0.$$  This precisely gives us an oriented four-cycle.\end{proof}

In the above proof we use the property $(\star)$, together with the trick of changing the order of summation, to deduce the existence of a pair of vertices $x,u$ between which there is a path of length two in each direction.  This is not a rare occurrence, rather it is typical.  Further, it is typical that the number of paths of length two in the two directions is of the same order.  The following lemma proves this fact and allows us to deduce Theorem \ref{fourcycles}.  Recall that $d^{+-}(x,u)$ denotes the number of paths of length two from $x$ to $u$.  Say that $(x,u)$ is \emph{unbalanced} if $d^{+-}(x,u)> 16 d^{+-}(u,x)$, otherwise it is \emph{balanced}.  A path of length two $\vec{xy},\vec{yu}$ is called \emph{unbalanced} if $(x,u)$ is unbalanced, otherwise it is \emph{balanced}.

\begin{lemma}\label{balance} Let $D$ be an oriented graph with $bias(D)<\ep e(D)^{2}/n^{2}$. Then the number of unbalanced paths of length two in $D$ is at most $8 \ep e(D)^{2}/n$.\end{lemma}

\begin{proof}[Proof of Theorem \ref{fourcycles}] Let $c=1/4.16^3$ and $\ep=1/128$.  By Lemma \ref{twopaths}, there are at least $e(D)^{2}/8n$ paths of length two in $D$.  By Lemma \ref{balance}, at most $8 \ep e(D)^{2}/n \le e(D)^{2}/16 n$ of these paths are unbalanced.  Therefore there are at least $e(D)^{2}/16n$ balanced paths in $D$.  Let $C_{x}$ denote the number of oriented four-cycles containing the vertex $x$. 
\begin{equation*} C_{x}=\sum_{u\in V} d^{+-}(x,u)d^{-+}(x,u)\ge \frac{1}{16}\sum_{u: (x,u)\, \text{balanced}}d^{+-}(x,u)^{2}\, .
\end{equation*}
Summing this quantity over $x$, and applying the Cauchy-Schwarz inequality, one obtains 
\begin{equation*}\sum_{x}C_{x}\ge \frac{1}{16}\sum_{x,u: (x,u)\, \text{balanced}}d^{+-}(x,u)^{2}\ge \frac{1}{16n^{2}}\Big(\sum_{(x,u) \,\text{balanced}}d^{+-}(x,u)\Big)^{2} .\end{equation*}
This sum counts exactly the number of balanced paths of length two and so is at least $e(D)^{2}/16n$.  Thus, $\sum_{x}C_{x}\ge e(D)^{4}/16^3n^{4}$.  The proof is now complete as this sum counts each oriented four-cycle four times.  \end{proof}

\begin{proof}[Proof of Lemma \ref{balance}] Denote by $f$ the number of unbalanced paths of length two, and suppose that $f\ge 8 \ep e(D)^{2}/n$.  Let $f_{x}$ denote the number of unbalanced paths of length two starting at $x$, and let $W$ denote the set of vertices $x$ for which $f_{x}\ge f/2n$.  Since at most $f/2$ unbalanced paths of length two start at vertices outside of $W$, we have that at least $f/2$ unbalanced paths start inside $W$,
 i.e., $\sum_{x\in W}f_{x}\ge f/2$.  For each vertex $x\in W$, we denote by $U_{x}$ the set of vertices $u$ for which $(x,u)$ is unbalanced, we have that $e(\Gamma^{+}(x),U_{x})=f_{x}\ge \ep e(D)^{2}/n^{2}$.  By $(\star)$ we have that $e(U_{x},\Gamma^{+}(x))\ge f_{x}/2$, 
i.e., $\sum_{u :\, (x,u)\, \text{unbalanced}}d^{++}(x,u)\ge f_{x}/2$.  Thus,
\begin{equation*} \sum_{(x,u) \text{unbalanced}}d^{++}(x,u)\ge \frac{1}{2}\sum_{x\in W}f_{x}\ge f/4\, .\end{equation*}  
Alternatively, denoting by $X_{u}$ the set of vertices $x$ for which $(x,u)$ is unbalanced, 
\begin{equation*} \sum_{u}e(X_{u},\Gamma^{+}(u)) \ge f/4\, .\end{equation*} 
Let $U$ denote those vertices $u$ for which $e(X_{u},\Gamma^{+}(u)) \ge f/8n$, and note that 
\begin{equation*} \sum_{u\in U}e(X_{u},\Gamma^{+}(u)) \ge f/8\, .\end{equation*} 
However, for each $u\in U$, we have by $(\star)$ that $e(\Gamma^{+}(u),X_{u})\ge e(X_{u},\Gamma^{+}(u))/2$.  And so 
\begin{equation*} \sum_{u\in U}e(\Gamma^{+}(u),X_{u}) \ge f/16 \, .\end{equation*} 
Reinterpreting this sum in terms of $d^{+-}(x,u)$ and recalling that 
$$f=\sum_{(x,u)\, \text{unbalanced}} d^{+-}(x,u),$$ we have that
\begin{equation*} \sum_{(x,u)\, \text{unbalanced}}d^{+-}(u,x) \ge \frac{1}{16} \sum_{(x,u)\, \text{unbalanced}}d^{+-}(x,u) \, .\end{equation*} 
Thus, there exists an unbalanced pair $(x,u)$ with $d^{+-}(u,x)\ge d^{+-}(x,u)/16$, a contradiction.\end{proof}

\section{Oriented Six-Cycles}\label{sixcyclessec}

In this section we prove Theorem \ref{sixcycles}.  
Note that this immediately implies Theorem \ref{sixcycle}.  First, a lemma showing that there are many paths of length two ending at vertices with out-degree at least $e(D)/8n$.

\begin{lemma}\label{goodtwopaths} Let $D$ be an oriented graph with $bias(D)\le e(D)/8$. Then $D$ contains at least $e(D)^{2}/8n$ oriented paths of length two whose end point has out-degree at least $e(D)/8n$.\end{lemma}

\begin{proof} We say that an arc $\vec{xy}$ is \emph{good} if $d^{-}(x)\ge d^{+}(x)/2$ and $d^{+}(y)\ge d^{-}(y)/2$.  It is \emph{very good} if in addition $d^{-}(y)\ge e(D)/4n$.  We first show that at least $e(D)/2$ arcs are very good.  Let $Z_{1}=\{x:d^{+}(x)>2d^{-}(x)\},\, Z_{2}=\{y:d^{-}(y)>2d^{+}(y)\}$ and $Z_{3}=\{y:d^{-}(y)<e(D)/4n\}$.  An arc is very good unless its start point is in $Z_{1}$ or its end point is in $Z_{2}\cup Z_{3}$.  Since $e(Z_{1},V)>2e(V,Z_{1})$ we have from $(\star)$ that $e(Z_{1},V)\le e(D)/8$.  Similarly $e(V,Z_{2})\le e(D)/8$.  Trivially $e(V,Z_{3})<e(D)/4$.  Hence, at least $e(D)/2$ arcs are very good.  For each vertex $x\in V\setminus Z_{1}$, let $d^{+}_{vg}(x)$ denote the number of very good arcs with start point $x$.  Thus, the number of paths of length two the second arc of which is very good, is at least 
\begin{equation*} \sum_{x\in V\setminus Z_{1}} d^{-}(x)d^{+}_{vg}(x)\ge \frac{1}{2}\sum_{x\in V\setminus Z_{1}}d^{+}_{vg}(x)^{2}\ge \frac{1}{2n}\Big (\sum_{x\in V\setminus Z_{1}} d^{+}_{vg}(x) \Big)^{2}\ge \frac{e(D)^{2}}{8n}\, ,\end{equation*}
where the second inequality follows from the Cauchy-Schwarz inequality.  The proof is now complete, as a path of length two whose second arc $\vec{xy}$ is very good has that $d^{+}(y)\ge d^{-}(y)/2\ge e(D)/8n$.\end{proof}

For our proof of Theorem \ref{sixcycles} we shall need some new notation.  
This notation will allow us to prove the existence of the constants $c,\ep>0$ 
without explicitly calculating them.  Specifically, $\gamma(\ep)$ denotes any 
decreasing function of $\ep$ which is positive for all sufficiently small $\ep>0$.  
While $\delta(\ep)$ denotes any increasing function of $\ep$ which has limit $0$ as $\ep\to 0$.    
It should be understood that $\gamma(\ep)$ and $\delta(\ep)$ are not functions, 
but rather classes of functions, in the sense that the commonly used $O(\cdot)$ notation represents 
a class of functions rather than an individual function. 
To give some examples to clarify this notation, 
we note that for two arbitrary constants $c_1, c_2 >0$, the function $-c_1 \epsilon +c_2$ is (of class)
$\gamma(\epsilon)$, while $-c_1 \epsilon$ is not $\gamma(\epsilon)$ since for all 
non-negative (even sufficiently small) 
values of $\epsilon$, $-c_1 \epsilon$ is negative. In the same way, for two arbitrary constants $c_1, c_2> 0$, the function $c_1\epsilon$ is (of class) 
$\delta(\epsilon)$ while the function $c_1\epsilon+c_2$ is not $\delta(\epsilon)$. 
It should be now clear that $\gamma(\ep)-\delta(\ep)=\gamma(\ep)$ and $\gamma(\ep)/C=\gamma(\ep)$ for any constant $C$, 
and also that $\gamma(\epsilon)\gamma(\epsilon)=\gamma(\epsilon)$.

In this terminology, to prove Theorem \ref{sixcycles} it suffices to show that every oriented graph $D$ with 
$bias(D)<\ep e(D)^{2}/n^{2}$ contains at least $\gamma(\ep)e(D)^{6}/n^{6}$ oriented six-cycles.

\noindent Let $D$ be such that $bias(D)<\ep e(D)^{2}/n^{2}$.  
We state some consequences of previous results in our new notation.  
Lemma \ref{goodtwopaths} tells us that $D$ contains at least $\gamma(\ep)e(D)^{2}/n$ paths of 
length two whose end point has out-degree 
at least $\gamma(\ep)e(D)/n$.  
While Lemma \ref{balance} tells us that $D$ contains 
at most $\delta(\ep)e(D)^{2}/n$ unbalanced paths of length two.  
Together, these results imply that $D$ contains $\gamma(\ep)e(D)^{2}/n$ paths of length two 
which are balanced and which have end points of out-degree at least $\gamma(\ep)e(D)/n$.  

\noindent We are now in position to start our proof.

\begin{proof}[Proof of Theorem \ref{sixcycles}]  Let $D$ be an oriented graph 
with $bias(D)<\ep e(D)^{2}/n^{2}$.  
Say that a path of length two which is balanced and whose end point has out-degree at 
least $\gamma(\ep)e(D)/n$ is a \emph{great} path.  
We found above that $D$ contains at least $\gamma(\ep)e(D)^{2}/n$ great paths.  
We now form a weighted directed graph $H$ on the same vertex set $V$ as follows. 
Start with the complete directed graph (i.e., the directed graph with arcs $\vec{xy}$ (and $\vec{yx}$) 
for all $x\neq y$), weight each arc $\vec{xu}$ by $w_{\vec{xu}}$ the number of great paths from $x$ to $u$, and delete arcs of weight zero.  
By the above result, the total weight of arcs in $H$ is at least $\gamma(\ep)e(D)^{2}/n$. 
Now divide the set of arcs of $H$ into two categories of heavy and light arcs: a heavy arc is an arc of 
weight at least $e(D)/n$ while a light arc has weight less than $e(D)/n$.  
Since the total weight in $H$ is at least $\gamma(\ep)e(D)^{2}/n$, 
either the total weight of heavy arcs or the total weight of the light arcs is at 
least $\gamma(\ep)e(D)^{2}/n$.
 We shall refer to these as Case I and Case II respectively, 
and denote by $H'$ the subgraph consisting of arcs of appropriate weight 
(heavy arcs of weight $\ge e(D)/n$ in Case I and light arcs of weight $<e(D)/n$ in Case II), 
although for the time being the proof continues for the two cases in parallel. 
 
For a weighted oriented graph and a vertex $v$, the out-weight (resp. in-weight) of $v$ is the total 
weight of the out-going (resp. in-coming) arcs of $v$.
Let $H''$ be the subgraph of $H'$ consisting of all the arcs 
whose start vertex has out-weight at least $\gamma(\ep)e(D)^{2}/n^{2}$. 
Since to total weight in $H'$ is at least $\gamma(\epsilon) e(D)^2/n$, a simple averaging argument shows that
there is a function $\gamma(\ep)$ such that the total weight of  arcs of $H''$ is 
at least $\gamma(\ep)e(D)^{2}/n$.  
For each vertex $x$ with positive out-weight in $H''$, 
we denote by $e_{x}$ its out-weight; in particular, note that by the definition of $H''$, 
$e_x \geq \gamma(\ep)e(D)^{2}/n^{2}$.  

\medskip
Our proof will be based on a local argument concerning the distribution of weights of $H''$ 
around each vertex of positive out-weight. To simplify the presentation, fix such a vertex $x$ and define a 
new (induced) weighting $\omega^x$, or simply $\omega$ if there is no risk of confusion, on the vertices  
by setting $\omega_{u}$ equal to the weight of 
the arc $\vec{xu}$ in $H''$ (if $\vec xu$ is not an arc of $H''$, simply define $\omega_u=0$). 
 Define the weight of a path of length two $\vec{uy},\vec{yv}$ in $D$ to be $\omega_{u}\omega_{v}$. 
We note again that the weighting $\omega$ depends on the choice of $x$ and is simply induced from the weighting of 
out-going arcs from $x$. 

{\bf Claim.} Under this weighting of the vertices, the total weight of paths of length two in $D$ is at least $\gamma(\ep)e_{x}^{2}e(D)^{2}/n^{3}$.  

Before presenting the proof of this Claim, and to keep the continuity of the proof, we show how the Claim 
allows to finish the proof of Theorem~\ref{sixcycles}. We will deduce from the Claim that there are at 
least $\gamma(\ep)e(D)^{6}/n^{6}$ oriented six-cycles in $D$.  
We recall that, for a fixed vertex $x$ and in terms of the original weighting of the arcs of $H''$, 
the Claim is discussing the total weight of paths of length two in $D$ where the weight of the path 
$\vec{uy},\vec{yv}$ is defined to be $w_{\vec{xu}}w_{\vec{xv}}$ 
(this now being the weight in $H''$; recall that $\omega_u = w_{\vec {xu}}$).  
Recall also that, by the definition of the weighting of the arcs, if $w_{\vec{xv}}>0$, 
then there are $w_{\vec{xv}}$ great paths of length two from $x$ to $v$ in $D$, 
and in particular this tells us that $(x,v)$ is balanced, 
and so that in addition there are at least $\gamma(\ep)w_{\vec{xv}}$ paths of 
length two from $v$ to $x$. This in turn implies that each path of length two 
$\vec{uy},\vec{yv}$ allows us to find $\gamma(\ep)w_{\vec{xu}}w_{\vec{xv}}$ oriented six-cycles 
containing $x$ (simply combine each of the $w_{\vec{xu}}$ paths of length two from $x$ to $u$ 
with the path $\vec{uy},\vec{yv}$ followed by each of 
the $\gamma(\ep)w_{\vec{xv}}$ paths of length two from $v$ to $x$).  Thus, writing $C_{x}$ for the number of oriented six-cycles in $D$ which contain $x$ and using the Claim, we have that $C_{x}\ge \gamma(\ep)e_{x}^{2}e(D)^{2}/n^{3}$.  Summing $C_{x}$ over all vertices $x$ with positive out-weight in $H''$, we obtain that $\sum_{x}C_{x}\ge \gamma(\ep)\sum_{x} e_{x}^{2}e(D)^{2}/n^{3}$.  By Cauchy-Scwartz, this is at least $\gamma(\ep)(\sum_{x}e_{x})^{2}e(D)^{2}/n^{4}$, and so, since the sum expresses the total weight of the arcs of $H''$, $\sum_{x}C_{x}\ge \gamma(\ep)e(D)^{6}/n^{6}$. This sum counts each oriented six-cycle at most six times, so the number of oriented six-cycles in $D$ is $\gamma(\ep)e(D)^{6}/n^{6}$,
and Theorem~\ref{sixcycles} follows.

Thus, we are only left to prove the above Claim. 

{\it Proof of Claim.} We recall that the weighting of vertices 
(induced by the fixed vertex $x$ of positive out-weight) is simply defined in such a way that 
the weight $\omega_u$ of a vertex $u$ is the weight $w_{\vec{xu}}$ of the arc $\vec xu$. 
We will divide the proof into two parts, depending on whether we are in Case I or in Case II. 

\noindent (Recall that in Case I, the subgraph $H'$ consists of all the heavy arcs, i.e., all the arcs 
of weight $\geq \gamma(\ep)e(D)/n$, while
 in Case II, $H'$ is the subgraph 
containing all the light arcs.)

The proof is simpler if we are in Case I.  In this case, 
every vertex with positive weight has weight at least $e(D)/n$, so that every path of length two 
with positive weight has weight at least $e(D)^{2}/n^{2}$. It thus suffices to find 
at least $\gamma(\ep)e_{x}^{2}/n$ paths of length two of positive weight.  
We denote by $U$ the set of vertices of positive weight, i.e., 
\[U:=\{u\in V(D)\,|\, \omega_u >0\}.\] 
Since $x$ has positive out-weight in $H''$, by the definition of $H''$, it has out-weight $e_{x}\ge
\gamma(\ep)e(D)^{2}/n^{2}$. 
 This implies that $e(\Gamma^{+}(x),U)=e_{x}\ge \gamma(\ep)e(D)^{2}/n^{2}$ and so, trivially, 
$e(V,U)\ge e_{x}\ge \gamma(\ep)e(D)^{2}/n^{2}$. Let now $Z$ be the set of all vertices which are 
``unbalanced`` with respect to $U$, namely, let  
$$Z:=\{y\,\,|\,\,e(\{y\},U)>2e(U,\{y\})\}.$$

Our condition on $bias(D)$ implies that $e(Z,U)<\ep e(D)^{2}/n^{2}]$.  
We infer that $e(Y,U)\ge \gamma(\ep)e_{x}$, where $Y=V\setminus Z$, 
and so the number of paths of length two from $U$ to $U$ is at least 
\begin{equation*} \sum_{y\in Y}e(\{y\},U)e(U,\{y\})\ge \frac{1}{2}\sum_{y\in Y}e(\{y\},U)^{2}\ge \frac{1}{2n}\Big(\sum_{y\in Y}e(\{y\},U)\Big)^{2}\ge \gamma(\ep) e_{x}^{2}/n\, .\end{equation*}  
Which completes the proof in Case I.  

\medskip

 For Case II, we divide the vertex set according to weights. For each $i\ge 1$, define 
$$V_{i}:=\Bigl\{\,\,u\,\, |\,\, \omega_{u}\in \bigl[\frac{e(D)}{2^{i}n}\,\,,\,\frac{e(D)}{2^{i-1}n}\bigr)\,\,\Bigr\}.$$ 
 Since we are in Case II, all the weights $\omega_u$ are less than $e(D)/n$ and thus,  
$V_+=\bigsqcup_{i\ge 1}V_{i}$, where $V_+$ is the set of all the vertices of positive weight.

The idea in this case is roughly speaking as follows 
(to make this idea work, we have to restrict the subset of indices to some subset $I$, see below): 
To prove the claim, that the total weight of paths of length two in $D$ 
with respect to the weighting $\omega$ is at least $\gamma(\ep)e_{x}^{2}e(D)^{2}/n^{3}$, 
we will only consider a subset of all the paths of length two from $V_+$ to $V_+$, and decompose the set 
of all such paths into different sorts depending on the starting point of the path being in 
$V_i$ and the end point of the path being in $V_j$, for $i,j \in I$ (for a subset $I$ of the index set that we will
define below). Let $e_2(V_i,V_j)$ be the number of paths of length two from $V_i$ to $V_j$. By the definition of $V_i$, 
all the weights $\omega_u$ are ``almost'' uniform for $u \in V_i$, namely,
 $\frac{e(D)}{2^{i}n}\leq \omega_u < \frac{e(D)}{2^{i-1}n}$. It follows that the total contribution 
of the paths of length two from $V_i$ to $V_j$, for $i,j\in I$, is 
$\Theta\Bigl(\,\sum_{i,j\in I}\frac{e(D)^2}{2^{i+j-2}n^2}e_2(V_i,V_j)\Bigr)$, and thus, we only need to show 
that $\sum_{i,j\in I}\frac{e(D)^2}{2^{i+j-2}n^2}e_2(V_i,V_j) \geq \gamma(\ep)e_{x}^{2}e(D)^{2}/n^{3}$. 
This is exactly what we will do in the following.

\medskip

For each $i$, let $s_{i}=|V_{i}|$.  Since by our assumption, $x$ 
has positive out-weight in $H''$, it must have out-weight $e_{x}\ge
\gamma(\ep)e(D)^{2}/n^{2}$.  This implies that 
\[e(\Gamma^{+}(x),V_{+})=e_{x}\ge \gamma(\ep)e(D)^{2}/n^{2}.\]  
Given that for each $i$, $\omega_u$ is at most $\frac{e(D)}{2^{i-1}n}$, a simple summation shows that
 at most $c e(D)^{2}/n^{2}$ of these arcs go to sets $V_{i}$ for which $s_{i}< ce(D)/n$, 
this for any constant $c>0$. We infer, by the definition of the class $\gamma(\ep)$, 
that for some sufficiently small constant $c$, we have 
\[ e(\Gamma^{+}(x),V_{I})\ge \gamma(\ep)e_{x},\]
 where $I=\Bigl\{i\,\,|\,s_{i}\ge ce(D)/n\Bigr\}$ and $V_{I}=\bigcup_{i\in I}V_{i}$.

\noindent Since $2\sum_{i\in I}s_{i}\frac{e(D)}{2^{i}n}\ge e(\Gamma^{+}(x),V_{I})$, it follows that 
\begin{equation}\label{sumi}\sum_{i\in I}\frac{s_{i}e(D)}{2^{i-1}n}\ge \gamma(\ep)e_{x}\, .\end{equation}

For every vertex $u\in V_+$, the arc $\vec{xu}$ has positive weight in $H''\ssq H$, and by the definition of 
$H$, we know that $u$ must have out-degree at least $\gamma(\ep)e(D)/n$ in $D$.  
Therefor, for each $i\in I$, one has  
\[e(V_{i},V)= \sum_{u \in V_i} d^+(u) \,\geq\, |V_i|\,\gamma(\epsilon)\,e(D)/n \,\geq\, c\gamma(\ep)e(D)^{2}/n^{2} = \gamma(\ep)e(D)^{2}/n^{2}.\] 
Define $Y_i$ to be the set of ``balanced'' vertices with respect to $V_i$, namely, 
\[Y_{i}=\Bigl\{\,y\,\,|\,\,e(V_{i},\{y\})\le 2e(\{y\},V_{i})\,\Bigr\}.\]
Using the definition of $bias$, we may bound by $\ep e(D)^{2}/n^{2}$ the number of arcs going to 
vertices $y\not\in Y_i$, deducing that 
\begin{equation}\label{eq:2}
e(V_{i},Y_{i})\ge \gamma(\epsilon)s_ie(D)/n\,\, \bigl(\geq \gamma(\ep)e(D)^{2}/n^{2}\bigr). 
\end{equation}
 
Consider the weighted bipartite simple graph $F$ which has parts $I$ and $V$, 
has an edge $iy$ whenever $y\in Y_{i}$, and has a weight $\eta_{iy}$ on the edge $iy$ 
given by $\eta_{iy}=\min\bigl\{\,e(\{y\},V_{i})\,,\,e(V_{i},\{y\})\,\bigr\}$.  Now for each walk of 
length two $iy,yj$ in $F$, we know there are 
at least $\eta_{iy}\eta_{jy}$ paths of length two from $V_{i}$ to $V_{j}$ in $D$ passing through $y$
(there paths are constructed by $e(V_i,\{y\}) \geq \eta_{iy}$ from $V_i$ to $y$ 
and $e(\{y\},V_j)\ge \eta_{jy}$ arcs from $y$ to $V_j$). 
 It follows that $e_2(V_i,V_j) \geq \sum_{y\in V} \eta_{iy}\eta_{jy}$.

Thus, to finish the proof of the claim, given that the weight $\omega_u$ of a vertex $u$ in $V_i$ is in 
the interval 
$ \bigl[\frac{e(D)}{2^{i}n}\,\,,\,\frac{e(D)}{2^{i-1}n}\bigr)$, 
and that the total weight of paths of length two from 
$\bigsqcup_{i\in I} V_i$ to $\bigsqcup_{i\in I} V_i$ 
is $\Theta\Bigl(\,\sum_{i,j\in I}\frac{e(D)^2}{2^{i+j-2}n^2}e_2(V_i,V_j)\Bigr)$, it suffices to prove that 

\begin{equation}\label{eq:3}
\sum_{i,j\in I}\sum_{y\in V}\,\, \frac{e(D)^2}{2^{i+j-2}n^2}\eta_{iy}\eta_{jy}\,\, \geq\,\, \gamma(\ep)e_{x}^{2}e(D)^{2}/n^{3}.
\end{equation}

Define the weight of a walk $iy,yj$ in $F$ to be  $e(D)^{2}\eta_{iy}\eta_{jy}/2^{i+j-2}n^{2}$. 
so that we will have to show that the total weight of walks of length two in $F$ is at least 
$\gamma(\ep)e_{x}^{2}e(D)^{2}/n^{3}$. 
This follows immediately from the more general result 
(itself a trivial consequence of the Cauchy-Schwarz inequality) 
that for any graph with weighted vertices and edges, the total weight of walks of length 
two (where the walk $uy,yv$ is assigned weight $\eta_{u}\eta_{uy}\eta_{yv}\eta_{v}$) is at least $W^{2}/n$, 
where $W$ denotes $\sum_{u}\eta_{u}\sum_{y}\eta_{uy}$.  In our case, in the bipartite graph $F$, 
the edge weights are $\eta_{iy}$ for $i\in I$ and $y\in Y_i$, and the vertex weights are defined by 
$\eta_i = \frac{e(D)}{2^{i-1}n}$ for $i\in I$ and $\eta(y) = 0$ for $y \in V$. We have
 \begin{align*}
W&=\sum_{i\in I}\eta_{i}\sum_{y}\eta_{iy}\\
&\geq \sum_{i\in I}\frac{e(D)}{2^{i-1}n}\,e(V_{i},Y_i)&(\textrm{by the definition of }\eta_{iy})\\
& \ge \frac {\gamma(\epsilon)e(D)}{n}\,\, \sum_{i\in I}\frac {s_i\,e(D)}{2^{i-1}n}, 
& (\textrm{by Inequality~(\ref{eq:2})})\\
& \ge \gamma(\epsilon)e_xe(D)/n & (\textrm{by Inequality~(\ref{sumi}}))
 \end{align*}
 The left term of~(\ref{eq:3}) is at 
least $W^2/n \geq \gamma(\ep)^2e_x^2e(D)^2/n^3 = \gamma(\ep)e_x^2e(D)^2/n^3$ and Inequality~(\ref{eq:3}) follows. This completes the proof of the Claim (and so the proof of the theorem).
\end{proof}

\section{Examples}\label{bpsec} 

To prove that a result such as Theorem \ref{fourcycle} is best possible up to the choice of constant,  it would not suffice to produce just one oriented graph which does not contain an oriented four-cycle and has $bias(D)=Ke(D)^{2}/n^{2}$, for some constant $K$.  Nor would it suffice to produce a class of examples that were all of the same density.  The theorem applies across a large range of densities, so one must produce examples across a large range of densities.  We provide in this section a wide class of examples of oriented graphs which do not contain oriented four-cycles and which have $bias(D)\le Ke(D)^{2}/n^{2}$, where $K$ is some fixed constant.

Our initial examples are obtained as random orientations of four-cycle free simple graphs.  These will have approximately $n^{3/2}$ arcs.  We will then obtain more dense examples as blow ups of these initial examples. 
 Before we do this, we first prove a lemma concerning the value of $bias(D)$ (and certain variants) for randomly oriented graphs, prove a lemma concerning the value of $bias(D)$ when $D$ is blow up of some other oriented graph, and recall a result concerning four-cycle free simple graphs.

We define a more general concept of $bias$.  For $\gamma\in(0,1)$, we say that a subgraph $E(A,B)$ is $\gamma$-biased if $e(B,A)\le \gamma e(A,B)$, and we write $bias_{\gamma}(D)$ for the size of the largest $\gamma$-biased subgraph of $D$, so that, \begin{equation*} bias_{\gamma}(D)=\max\{e(A,B):A,B\ssq V\quad \text{with}\quad e(B,A)\le \gamma e(A,B)\} \, .\end{equation*} Note that $bias(D)$ is of
course $bias_{1/2}(D)$.

\begin{lemma}\label{random} Given $\gamma \in (0,1)$, there exists $K_\gamma \in
\mathbb{R}$ such that for every simple graph $G$ on $n$ vertices, there exists
an oriented graph $D$ obtained by orienting the edges of $G$ with $bias_{\gamma}(D)<K_\gamma
n$.  Furthermore, a random orientation $D$ of $G$ has $bias_{\gamma}(D)<K_\gamma
n$ with high probability.
\end{lemma}

\begin{proof}

For a pair $A,B\ssq V(G)$ we write $e_{AB}$ for the number of edges between
$A$ and $B$ in the graph $G$. Let $D$ be obtained from $G$ by orienting its
edges at random.  From Chernoff's inequality \cite{Ch} (see Lemma \ref{Chernoff}) we have \begin{displaymath}\mathbb{P}(e(B,A)\le \gamma e(A,B))\le \exp(-c(\gamma)e_{AB})\le
\exp(-c(\gamma)K_\gamma n)\, , \end{displaymath} where $c(\gamma)$ is a positive constant
dependent on $\gamma$. We set $K_{\gamma}=2/c(\gamma)$.  
We shall prove that $\mathbb{P}(bias_{\gamma}(D)\ge K_{\gamma}n)\le 1.8^{-n}$. This proves the lemma.  The event $bias_{\gamma}(D)\ge K_{\gamma}n$ can occur only if there is a pair $A,B\ssq V(G)$ with $e_{G}(A,B)\ge K_{\gamma}n$ for which $e(B,A)\le \gamma e(A,B)$.  There are at most $4^{n}$ such pairs $(A,B)$ and for each such pair the probability that $e(B,A)\le \gamma e(A,B)$ is at most $\exp(-c(\gamma)e_{AB})\le \exp(-c(\gamma)K_{\gamma}n)=\exp(-2n)$.  Thus, by the union bound, $\mathbb{P}(bias_{\gamma}(D)\ge K_{\gamma}n)\le 4^{n} \exp(-2n)\le 1.8^{-n}$.\end{proof}

If an oriented graph $D'$ contains no large biased subgraphs, then this property carries over, in a weakened form, to a blow-up $D$ of $D'$.  A blow-up of an oriented graph $D'$ is defined as follows.

\begin{definition}\rm Let $D'$ be an oriented graph on $\{1,\dots ,m\}$ and let $l\in \mathbb{N}$. 
The $l$-{\it blow-up} of $D'$ is 
the oriented graph $D$ with vertex set $V=V_{1}\cup\dots \cup V_{m}$, 
where the sets $V_{i}$ are disjoint and each of cardinality $l$, 
and with arc set $E(D)=\cup_{ij\in E(D')}B(V_{i},V_{j})$, 
where $B(V_{i},V_{j})$ represents 
the complete bipartite oriented graph on $V_{i}\cup V_{j}$ 
with all arcs going from $V_{i}$ to $V_{j}$. 
Each $V_i$ is called a {\it cell} of the blow-up.\end{definition}

The key result we need about blow-ups is,

\begin{lemma}\label{blowup} If $bias_{0.9}(D')<f$ and $D$ is an $l$-blow-up of $D'$, then $bias(D)<16fl^{2}$.\end{lemma}

\begin{proof} Let $D$ be an $l$-blow-up of $D'$.  We suppose that $bias(D)\ge 16fl^{2}$ and use this to show that $bias_{0.9}(D')\ge f$.  Our assumption gives us that there exist sets $A,B\ssq V(D)$ with $e(A,B)\ge 16fl^{2}$ and $e(B,A)\le e(A,B)/2$.  We use this irregularity between $A$ and $B$, this bias in the direction from $A$ to $B$, to find subsets $I,J\ssq \{1,\dots ,m\}$ such that in $D'$ we have $e_{D'}(I,J)\ge f$ and $e_{D'}(J,I)\le 0.9 e(I,J)$, this will prove $bias_{0.9}(D')\ge f$ and so will complete the proof.

As a warm-up, we first consider the easier case when $A$ and $B$ are both unions of cells of the blow-up.  
In this case $A=\cup_{I}V_{i}$ and $B=\cup_{J}V_{j}$.  Which gives \begin{equation*} e(I,J)=\frac{e(A,B)}{l^{2}}\ge \frac{16fl^{2}}{l^{2}}=16f \ge f \, ,\end{equation*} while \begin{equation*} e(J,I)=\frac{e(B,A)}{l^{2}}\le \frac{e(B,A)}{2l^{2}}=\frac{e(I,J)}{2}\, .\end{equation*}

Now for the general case, if there is a relative deficiency in the number of arcs from $B$ to $A$, we identify the vertices of $A$ responsible for this.  Let \begin{equation*} A'=\Big\{x\in A:e(\{x\},B)\ge
\frac{3}{2}e(B,\{x\})\Big\}\, .\end{equation*}  Note that $e(B,A\setminus A')\ge
2e(A\setminus A',B)/3$ so that \begin{equation*} e(A,B)\ge
2e(B,A)\ge 2e(B,A\setminus A')\ge \frac{4}{3}e(A\setminus A',B)\, .
\end{equation*} 
Thus, $e(A\setminus A',B)\le 3e(A,B)/4$,
and so $e(A',B)\ge e(A,B)/4\ge 4f l^{2}$.  Now, let
$I=\{i:V_{i}\cap A'\not = \phi \}$ and let $A''=\cup_{I}V_{i}$.  By the homogeneity of parts of the blow-up, we have for all $x\in A''$ that \begin{equation*} e(\{x\},B)\ge \frac{3}{2}e(B,\{x\})\, .\end{equation*}  Also $e(A'',B)\ge e(A',B)\ge 4fl^{2}$.  We now begin a similar procedure to find $B''$. We let \begin{equation*} B'=\Big\{y\in B:e(A'',\{y\})>
\frac{10}{9}e(\{y\},A'')\Big\}\, .\end{equation*} This implies that $e(B\setminus B',A'')\ge
9e(A'',B\setminus B')/10$, so that \begin{equation*} e(A'',B)\ge
\frac{3}{2}e(B,A'')\ge \frac{3}{2} e(B\setminus B',A'')\ge
\frac{27}{20}e(A'',B\setminus B')\, . \end{equation*} 
Thus,
$e(A'',B\setminus B')\le 20e(A'',B)/27 $, and so \begin{equation*}e(A'',B')\ge
\frac{7}{27}e(A'',B) \ge \frac{1}{4}e(A'',B) \ge fl^{2} \, .\end{equation*}  Let
$J=\{j:V_{j}\cap B'\not = \phi\}$, and set $B''=\cup_{J} V_{j}$. With a similar
argument to that given previously we obtain that $e(A'',B'')\ge e(A'',B')\ge
fl^{2}$ and 
\begin{equation*}e(B'',A'')< \frac{9}{10}e(A'',B'')\, .\end{equation*} 
This tells us that in $D'$ we have $e(I,J)\ge fl^{2}/l^{2} = f$, while 
\begin{equation*} e(J,I)=\frac{e(B'',A'')}{l^{2}}<\frac{9e(A'',B'')}{10l^{2}}=\frac{9e(I,J)}{10}\, .\end{equation*} 
\end{proof}

The final piece of information we need before stating our examples concerns 
the existence of large four-cycle free simple graphs.  Let $q$ be a prime
power.  The Erd\H os-R\'enyi graph $G$ \cite{ERS} has
$V(G)$ being the set of points of the finite projective plane $PG(2,q)$ over
the field of order $q$ (so that $n=q^{2}+q+1$), and an edge between $(x,y,z)$
and $(x',y',z')$ if and only if $xx'+yy'+zz'=0$. In fact this implies
$e(G)=\frac{1}{2}q(q+1)^{2}$ and this graph does not contain a four-cycle. So
for
all $n$ of the form $q^{2}+q+1$ (where $q$ is a prime power), there is a graph
$G$ on $n$ vertices with at least $\frac{1}{2}n^{3/2}$ edges which does not
contain a four-cycle. We would like a result which holds for all $n$. We recall
that Bertrand's Postulate states that for all $k\ge 2$, there is a prime between
$k$ and $2k$, combining this with the above example one may deduce the following.

\begin{lemma} Given $n\ge 2$, there exists a graph $G$ on $n$ vertices with at
least $\frac{1}{20}n^{3/2}$ edges which does not contain a four
cycle.\end{lemma}\vspace{-0.5cm}
\qquad\hspace*{\fill}$\Box$

\vspace{.3cm}

We may now state examples of oriented graphs with $bias(D)<Ke^{2}/n^{2}$ which do not contain an oriented four-cycle.  Our first examples are obtained by considering random orientations of four-cycle free simple graphs.  By the above lemma, there exists, for each $n$, a simple graph $G$ on $n$ vertices which is four-cycle free and has $e(G)\ge n^{3/2}/20$.  Let $D$ be obtained by orienting the edges of $G$ at random.  Then $D$ certainly cannot contain an oriented four-cycle and, by Lemma \ref{random}, with positive probability  $bias(D)\le K_{1/2}n\le 400K_{1/2}e(D)^{2}/n^{2}$.  In particular this gives us for all $n$ an oriented graph on $n$ vertices which does not contain an oriented four-cycle and for which $bias(D)\le 400K_{1/2}e(D)^{2}/n^{2}$.

Our more general class of examples is obtained by considering blow-ups of the above examples.  For a fixed constant $K$ (in fact we take $K=6400 K_{0.9}$), we define for each pair of natural numbers $m$ and $l$ an oriented graph $D_{m,l}$ which contains no oriented four-cycle and has $bias(D)\le K e(D)^{2}/n^{2}$.  The number of vertices of $D$ will be $n=ml$, while $e(D)$, the number of arcs of $D$, will be of the order $m^{3/2}l^{2}$.  

\noindent Fix a pair of natural numbers $m$ and $l$.  Let $G$ be a four-cycle free simple graph on $m$ vertices with at least $m^{3/2}/20$ edges. Let $D'$ be an oriented graph obtained by orienting $G$ and such that $bias_{0.9}(D')<K_{0.9}m$, the existence of such an oriented graph being assured by Lemma \ref{random}.  Let $D=D_{m,l}$ be obtained as an $l$-blow-up of $D'$.  We now have, by Lemma \ref{blowup}, that $bias(D)< 16K_{0.9}ml^{2}$.  It is easily observed that $n=ml$ and $e\ge m^{3/2}l^{2}/20$, so that $ml^{2}<400 e(D)^{2}/n^{2}$.  Thus, $bias(D)<6400K_{0.9}e(D)^{2}/n^{2}$, and, by inspection, $D$ does not contain any oriented four-cycle.

As a demonstration of the generality of our class $(D_{m,l})_{m,l\in\mathbb{N}}$ of examples, note that for any pair $n_{0},e_{0}$ with $e_{0}\ge n_{0}^{3/2}$, there is a choice of $m$ and $l$ such that $D_{m,l}$ has approximately $n_{0}$ vertices and approximately $e_{0}$ arcs.  Simply choose $m$ to be an integer close to $n_{0}^{4}/400e_{0}^{2}$, choose a four-cycle free graph $G$ with close to
$m^{3/2}/20$ edges, and choose $l$ to be close to $400e_{0}^{2}/n_{0}^{3}$.

\section{The case $D$ is dense}\label{densesec} 

In this section we prove Theorem \ref{dense}.  In fact we shall prove a more general result, Proposition \ref{denseprop}, which also counts homomorphic copies of partially oriented graphs.  A partially oriented graph $H$ is a graph which may have some of its edges oriented.  We write $\vec{e}(H)$ for the number of edges of $H$ that are oriented, e.g. if $H$ is a simple graph then $\vec{e}(H)=0$ and if $H$ is an oriented graph then $\vec{e}(H)=e(H)$.  We also introduce the notation $\bar{e}(A,B)$ for the total number of edges (whatever their orientation) between $A$ and $B$.  Note that if $D$ is an oriented graph with $bias(D)<\ep n^{2}$, then in particular the following holds in $D$:
\begin{equation}\label{epgives} e(B,A)\ge \frac{\bar{e}(A,B)}{3} - \frac{\ep}{3} n^{2}\qquad \text{for all}\, A,B\ssq V\, .\end{equation}
We now turn to Proposition \ref{denseprop}.  This proposition clearly implies Theorem \ref{dense}.

\begin{proposition}\label{denseprop} Let $D$ be an oriented graph on $n$ vertices satisfying (\ref{epgives}).  Let $H$ be a partially oriented graph on $k$ vertices.  Then \begin{equation*} hom(H,D)\ge \frac{hom(\bar{H},D)}{3^{\vec{e}(H)}} - (1-3^{-\vec{e}(H)})\frac{\ep}{2} n^{k}\, .\end{equation*}\end{proposition}
 
  \begin{remark}\rm
    To prove that a dense oriented graph with bias at most $\epsilon n^2$ contains $H$ as a subgraph, one may use Szemer\'edi Regularity Lemma (cf proof in Appendix of the arxiv version of this article (arXiv:0911.3969)). However we believe that the use of Proposition \ref{denseprop} with the following direct proof is simpler and illustrates better the use of the bias parameter. Furthermore it gives better constants.
  \end{remark}

\begin{proof} We prove the proposition by induction on $\vec{e}(H)$.  If $\vec{e}(H)=0$, then $H=\bar{H}$, and so $hom(H,D)=hom(\bar{H},D)$.  For the general case, let $H$ be an oriented graph on $\{1,\dots,k\}$ with $\vec{e}(H)\ge 1$.  By relabelling if necessary (which does not affect the homomorphism count) we may assume that $\vec{12}$ is an arc (oriented edge) of $H$.  Let $H'$ be the partially oriented graph obtained by unorienting this edge.  We now relate the quantities $hom(H',D)$ and $hom(H,D)$.  For each $(x_{3},\dots,x_{k})\in V^{k-2}$, let $Hom(H',D;\, .\, ,\, .\, ,x_{3},\dots,x_{k})$ denote the set of homomorphisms $\phi$ of $H'$ into $D$ for which $\phi(i)=x_{i}$ for all $i=3,\dots,k$.  Similarly define $Hom(H,D;\, .\, ,\, .\, ,x_{3},\dots,x_{k})$.  In fact, it is easy to characterise the homomorphisms $\phi \in Hom(H',D;\, .\, ,\, .\, ,x_{3},\dots,x_{k})$.  A homomorphism $\phi \in Hom(H',D;\, .\, ,\, .\, ,x_{3},\dots,x_{k})$ must have $\phi(i)=x_{i}$ for $i= 3,\dots,k$, and must pick values for $\phi(1)$ and $\phi(2)$.  Writing $x_{1}$ for $\phi(1)$, we know $x_{1}$ must join up appropriately to the vertices $x_{3},\dots,x_{k}$.  Specifically

\begin{tabular}{c p{1cm} p{10cm}} 

$(i)$ & & $\vec{x_{1}x_{i}}$ is an arc of $D$, for every arc $\vec{1i}:i\ge 3$ in $H$.
\\
$(ii)$ & & $\vec{x_{i}x_{1}}$ is an arc of $D$, for every arc $\vec{i1}:i\ge 3$ in $H$.
\\
$(iii)$ & & $x_{1}x_{i}$ is an edge of $\bar{D}$, for every edge $1i:i\ge 3$ in $\bar{H}$.

\end{tabular}

Equivalently, $x_{1}\in \bigcap_{i\ge 3:\vec{1i}\in E(H)}\Gamma^{-}(x_{i})\cap \bigcap_{i\ge 3:\vec{i1}\in E(H)}\Gamma^{+}(x_{i})\cap \bigcap_{i\ge 3:1i\in E(\bar{H})}\Gamma(x_{i})$.  We denote this set $A$.  Similarly, writing $x_{2}$ for $\phi(2)$, there are similar restrictions on $x_{2}$, which again are equivalent to demanding that $x_{2}$ belongs to a certain set, we denote this set $B$.  Since $H'$ has an unoriented edge between $1$ and $2$, we have a final condition - the condition that $x_{1}x_{2}$ is an edge of $\bar{D}$.  Hence for certain sets $A$ and $B$, we have a one-to-one correspondence between homomorphisms $\phi\in Hom(H',D;\, .\, ,\, .\, ,x_{3},\dots,x_{k})$ and edges of $\bar{D}$ between $A$ and $B$.

Similarly, we may characterise the homomorphisms $\phi\in Hom(H,D;\, .\, ,\, .\, ,x_{3},\dots,x_{k})$.  Again we write $x_{1}$ and $x_{2}$ for $\phi(1)$ and $\phi(2)$. The restrictions $x_{1}\in A$ and $x_{2}\in B$ remain.  However, on this occasion we require not only that there is some edge between $x_{1}$ and $x_{2}$, but that there is an oriented edge from $x_{1}$ to $x_{2}$.  Thus, there is a one-to-one correspondence between homomorphisms $\phi\in Hom(H,D;\, .\, ,\, .\, ,x_{3},\dots,x_{k})$ and edges from $A$ to $B$.

Thus, $|Hom(H',D;\, .\, ,\, .\, ,x_{3},\dots,x_{k})|$ and $|Hom(H,D;\, .\, ,\, .\, ,x_{3},\dots,x_{k})|$ are $\bar{e}(A,B)$ and $e(A,B)$ respectively, for some pair of subsets $A,B\ssq V$.  From our condition (\ref{epgives}), we obtain that \begin{equation*} |Hom(H,D;\, .\, ,\, .\, ,x_{3},\dots,x_{k})|\ge \frac{|Hom(H',D;\, .\, ,\, .\, ,x_{3},\dots,x_{k})|}{3} - \frac{\ep}{3} n^{2}\, .\end{equation*} Since $hom(H,D)$ is the sum over $(x_{3},\dots,x_{k})\in V^{k-2}$ of $|Hom(H,D;\, .\, ,\, .\, ,x_{3},\dots,x_{k})|$, and similarly $hom(H',D)$, we have that
\begin{equation*} hom(H,D)\ge  \frac{hom(H',D)}{3}- \frac{\ep}{3} n^{k}\, .\end{equation*} 
Having obtained this relation between $hom(H,D)$ and $hom(H',D)$, we require only an application of the induction hypothesis.  As $\vec{e}(H')=\vec{e}(H)-1$, an application of the induction hypothesis to $H'$ gives $hom(H',D)\ge hom(\bar{H},D)/3^{\vec{e}(H)-1}-(1-3^{1-\vec{e}(H)})\frac{\ep}{2} n^{k}$.  Combining this with the inequality proved above
\begin{equation*} hom(H,D)\ge  \frac{hom(\bar{H},D)}{3^{\vec{e}(H)}}-\frac{1}{3}\Big(1-3^{1-\vec{e}(H)}\Big)\frac{\ep}{2}n^{k}-\frac{\ep}{3} n^{k}= \frac{hom(\bar{H},D)}{3^{\vec{e}(H)}} - (1-3^{-\vec{e}(H)})\frac{\ep}{2} n^{k} \, .\end{equation*}
\end{proof}

\section{The extremal problem concerning $bias(D)$}\label{regran}

The main focus of the present article is the question of what structural information on an oriented graph $D$ can be obtained from the knowledge that a certain oriented graph $H$ is not a subgraph of $D$.  More precisely, our results relate to the question:
\begin{quote}What is the minimum value of $bias(D)$ over $H$-free oriented graphs $D$ with 
$n$ vertices and $e$ arcs?\end{quote}  
In this section we consider the related extremal question: 
\begin{quote}What is the minimum value of $bias(D)$ over oriented graphs on $n$ (non-isolated) 
vertices?\end{quote}  
It will be seen from the results we obtain that this is indeed the most natural form of the extremal question.

The definition of $bias(D)$ contains the constant $\frac{1}{2}$.  As commented in the introduction, this choice is rather arbitrary, and all our results hold (up to a change of constant) if $\frac{1}{2}$ is replaced by some other constant $\eta\in(0,1)$.  The same is true for the results of this section, but since it is not trivial to deduce the bounds on $bias_{\eta}$ from those for $bias=bias_{1/2}$, we shall prove our bounds for $bias_{\eta}$, $\eta\in (0,1)$.  We shall state our lower bounds as lower bounds on the quantity $ow(D):=\max\{e(A,B):A,B\subset V(D),\, e(B,A)=0\}$, 
the size of the largest one-way subgraph in $D$.  Since, trivially, $bias_{\eta}(D)\ge ow(D)$,
 this provides a lower bound on $bias_{\eta}(D),\, \eta\in (0,1)$.  Our results are as follows.

\begin{theorem}\label{lb} Every oriented graph $D$ on $n$ (non-isolated) vertices has $ow(D)\ge n/9\lceil \log_{2}{n}\rceil$.\end{theorem}

\begin{remark}\rm This result is not difficult to prove, and we remark that it may be easily 
deduced from Lemma 2 of \cite{FLS} (in fact the tight version of that lemma, 
whose proof is sketched after the proof of Lemma 2). \end{remark}

\begin{theorem}\label{biasex} For each $\eta\in (0,1)$ there is a constant $K_{\eta}$ such that, for every $n$, there is an oriented graph $D$ on $n$ (non-isolated) vertices with $bias_{\eta}(D)\le K_{\eta}n/\log_{2}{n}$.\end{theorem}

\begin{remark}\rm This result is proved by considering a random orientation of an 
inhomogeneous random graph.  The resulting graphs typically contain $\Theta(n^{2}/(\log{n})^{2})$ edges.  Variants of this example could be produced with $\Theta(n^{\alpha})$ edges, for any $\alpha\in (1,2)$ (although the constant $K_\eta$ would be dependent on $\alpha$).\end{remark}

One surprising facet of our results is that a much larger one-way subgraph can be found in the case that $D$ is out-regular (all out-degrees equal), irrespective of the degree: the lower bound $n/4$ follows immediately from the following proposition.

\begin{proposition}\label{realreg} Let $D$ be an oriented graph with maximum out-degree $\Delta^{+}$, then $ow(D)\ge e(D)/4\Delta^{+}$.\end{proposition}

For oriented graphs that are close to out-regular, in the sense that $\Delta^{+}$ is at most a constant multiple of the average out-degree, the above bound is of the same order as the upper bound (Lemma \ref{random}) on $bias_{\eta}(D)$ for $D$ whose orientation is random.  Simple applications of Chernoff bounds \cite{Ch} (see Lemma \ref{Chernoff}) imply that a random orientation of the Erd\H os-R\' enyi random graph $G(n,p)$ (with $p=\omega(\log{n}/n)$) will be close to out-regular with high probability.  From which we may deduce.
 
\begin{corollary} Let $D$ be a random orientation of the Erd\H os-R\' enyi random graph $G(n,p)$, with $p=\omega(\log{n}/n)$.  Then $ow(D)=\Theta(n)$ with high probability.\end{corollary}

We also note the following cute lower bound, which follows from Lemma \ref{realreg} and the trivial lower bound $ow(D)\ge \Delta^{+}$.

\begin{proposition}\label{sqrt} Let $D$ be an oriented graph, then $ow(D)\ge \sqrt{e(D)}/2$.\end{proposition}

Since there are many results in this section, we shall split into subsections.  
In subsection \ref{lbs} we prove the lower bounds on $ow(D)$, Theorem \ref{lb} and Proposition \ref{realreg}.  
In subsection \ref{biasexsec} we prove Theorem \ref{biasex}.  
We finish the section with an algorithmic proof of the lower bound $ow(D)\ge n/4$ 
for regular oriented graphs $D$.

We also draw here the reader's attention to the article of Brown, Erd\H os and Simonovits \cite{BES73}, 
which considers extremal problems for directed graphs, where parallel arcs in opposite directions are allowed.

\medskip

Let us state here, the form of Chernoff's Inequality \cite{Ch} that we shall use throughout the section.

\begin{lemma}[Chernoff bound for the sum of Poisson
trials]\label{Chernoff} Let $\beta\in [0,1]$ and let $S_{N}=X_{1}+...+X_{N}$ be a sum of independent
random variables $X_{i}$ taking values in $\{0,1\}$, with
$\mathbb{P}(X_{i}=1)=p_{i}$ for each $i=1,...,N$, so that the expectation of
$S_{N}$ is $\mu=\sum_{i}p_{i}$.  Then, for all $\lambda\ge \mu$, \begin{equation*} \mathbb{P}(S_{N}\ge
(1+\beta) \lambda)\le \exp\Big(\frac{-\beta^{2}\lambda}{3}\Big)\, .\end{equation*} In the other direction, for all $\lambda\le \mu$, \begin{equation*} \mathbb{P}(S_{N}\le (1-\beta)\lambda)\le \exp\Big(\frac{-\beta^{2}\lambda}{2}\Big)\, .\end{equation*}  \end{lemma}

The usual statement of Chernoff's inequality is the $\lambda=\mu$ case of the above statement, a proof of which is given in \cite{MU}.  The more general statement given here may be deduced from this by an easy monotonicity argument.

\subsection{Lower bounds on $ow(D)$}\label{lbs}

Throughout the section we shall write $B(A)$ for the set $\{v\in V:e(\{v\},A)=0\}$.  Since $B(A)$ maximises $e(A,B)$ over the set of all sets $B$ with $e(B,A)=0$ we have that $$ow(D)=\max_{A\ssq V(D)}e(A,B(A))\, .$$ All we shall need for the following proofs is that, for any probability distribution on subsets $A$ of $V(D)$, we have $ow(D)\ge \mathbb{E}(e(A,B(A)))$.

\begin{proof}[Proof of Proposition \ref{realreg}]  
Let $D$ be an oriented graph with maximum out-degree $\Delta^{+}$.  
Let $A\ssq V$ be a subset selected at random, each vertex included in $A$ 
independently with probability $p$.  To prove the proposition it suffices to show that, 
when $p$ is chosen appropriately, 
\begin{equation*} \mathbb{E}(e(A,B(A)))\ge \frac{e(D)}{4\Delta^{+}}\, .\end{equation*} 
Expressing $e(A,B(A))$ as $\sum_{\vec{xy}\in E}1_{x\in A, y\in B(A)}$, 
exchanging the order of summation and expectation, and using that $\mathbb{E}(1_{F})=\mathbb{P}(F)$ 
for any event $F$, we obtain 
\begin{equation*}\mathbb{E}(e(A,B(A)))=\sum_{\vec{xy}\in E}\mathbb{P}(x\in A, y\in B(A))\, .\end{equation*}
Fix an arc $\vec{xy}\in E(D)$.  The event $y\in B(A)$ is exactly 
the event $\bigcap_{v\in \Gamma^{+}(y)}(v\not\in A)$.  Thus, this event is independent of the 
event $x\in A$, and the probability that it occurs is at least 
$(1-p)^{d^{+}(y)}\ge (1-p)^{\Delta^{+}(D)}\ge 1-p\Delta^{+}(D)$.  From which we deduce,
\begin{equation*}\mathbb{E}(e(A,B(A)))\ge e(D) p(1-p\Delta^{+}).\end{equation*}
The proof is now completed by taking $p=1/2\Delta^{+}$.\end{proof}

\noindent In the case that $D$ is a regular oriented graph, 
the above proof shows that $\mathbb{E}(e(A,B(A)))\ge n/4$ where 
$A$ is selected by including each vertex in $A$ independently with probability 
$p=1/2d$ (where $d$ is the regular in- and out-degree). 
 In the general case, there may be a large range of degrees in $D$, 
and so there is no obvious choice of $p$.  To prove Theorem \ref{lb}, we first find a patch of 
vertices with approximately the same degree, and then proceed as above.

\begin{proof}[Proof of Theorem \ref{lb}] Define subsets $V^{+},V^{-}$ of the vertex set $V$ by setting $$V^{+}=\{v:d^{+}(v)\ge d^{-}(v)\}\quad \text{and}\quad V^{-}=\{v:d^{-}(v)\ge d^{+}(v)\}\, .$$
Partition $V^{+}$ into 
$V^{+}_{1},\dots V^{+}_{\lceil \log_{2}{n} \rceil }$ by defining 
$$V^{+}_{i}=\Bigl\{\,v\in V^{+}:\,d^{+}(v)\in [n/2^{i},n/2^{i-1})\,\Bigr\}\, .$$
Similarly partition $V^{-}$ into $V^{-}_{1},\dots ,V^{-}_{\lceil \log_{2}{n}\rceil}$.  
Since these  $2\lceil \log_{2}{n}\rceil$ sets cover $V$, one of them must have size 
at least $n/2\lceil \log_{2}{n}\rceil$.  By reversing the orientation of $D$ if necessary 
(which has no effect on $ow(D)$), we may assume that $|V^{-}_{i}|\ge n/2\lceil \log_{2}{n}\rceil$ 
for some $i$.  Having chosen this patch of vertices $V^{-}_{i}$, we define $p=2^{i}/3n$.  
We define the random subset $A$ of $V$ by including each vertex in $A$ independently 
with probability $p$.  Let $B=B(A)\cap V^{-}_{i}$.  
Arguing as in the proof of Proposition~\ref{realreg} above, we have, 
for each $y\in V^{-}_{i}$ and each arc $\vec{xy}$ of $D$, 
that the events $x\in A$, $y\in B$ are independent and that 
$\mathbb{P}(y\in B)=(1-p)^{d^{+}(y)}\ge (1-p)^{d^{-}(y)}\ge 1-pd^{-}(y)$.  Thus, 
$$ \mathbb{E}(e(A,B))=\sum_{y\in V^{-}_{i}}\sum_{\vec{xy}\in E(D)}\mathbb{P}(x\in A, y\in B)\ge \sum_{y\in V^{-}_{i}}pd^{-}(y)(1-pd^{-}(y))\, .$$ 
By the choice of $p$ we have that $pd^{-}(y)\in [\frac{1}{3},\frac{2}{3}]$.  
Since $x(1-x)\ge \frac{2}{9}$ for all $x\in [\frac{1}{3},\frac{2}{3}]$, 
we deduce that $$\mathbb{E}(e(A,B))\ge \frac{2|V^{-}_{i}|}{9}\ge \frac{n}{9\lceil \log_{2}{n}\rceil}\, .$$
\end{proof}

\subsection{Proof of Theorem \ref{biasex}: An oriented graph with no large $\eta$-biased subgraph}\label{biasexsec}

Fix $\eta\in (0,1)$.  To prove Theorem \ref{biasex} we must prove, 
for some constant $K=K_{\eta}$, that for every $n\in \mathbb{N}$, 
there is an oriented graph $D$ on $n$ (non-isolated) vertices 
with $bias_{\eta}(D)\le K n/\log_{2}{n}$.  For notational convenience, we in fact prove the 
bound $bias_{\eta}(D)\le K \lceil n/l\rceil $ where $l=\lfloor \log_2(n)\rfloor$.

We now define the random oriented graph $D$.

Let $V:=\{1,\dots ,n\}$.  We partition $V$ into subsets $V_{1},\dots ,V_{l}$ of approximately 
equal size.  Let $n=l\lfloor \frac nl\rfloor+q$, where $q$ is an integers $0\le q<l$. 
We denote $r:=\lceil \frac nl\rceil$.
 Let $V_{1},\dots ,V_{l}$ 
be a partition of $V$ such that:
$$ |V_{i}|=r \quad \text{for $i=1,\dots , q$}\quad \text{and} \quad |V_{i}|=\lfloor \frac nl \rfloor \quad \text{for $i=q+1,\dots ,l$}\, .$$
Let $G$ be the random simple graph on vertex set $V$ defined by including each edge $xy$, $x\in V_{i},\, y\in V_{j}$, in $G$ with probability $1/2^{i+j-1}$, independently of all other edges.  Let $D$ be obtained from $G$ by orienting its edges at random.

\noindent\textbf{Basic Properties of $D$.}\\
$(i)$ Each arc $\vec{xy}$, $x\in V_{i},\, y\in V_{j},\, i\neq j$, has probability $1/2^{i+j}$ of being an arc 
of $D$.\\
$(ii)$ The event $\vec{xy}\in E(D)$ is independent of the event $\vec{uv}\in E(D)$ 
so long as $\{u,v\}\neq \{x,y\}$.  Furthermore, the event $\vec{xy}\in E(D)$ is independent of 
all events $F$ generated by the events $\vec{uv}\in E(D):\{u,v\}\neq\{x,y\}$.\\
$(iii)$ While there is dependence between the events $\vec{xy}\in E(D)$ and $\vec{yx}\in E(D)$ 
(e.g., they are mutually exclusive), 
the bounds $0\le \mathbb{P}\bigl(\vec{xy}\in E(D)\,|\,F\bigr)\le 1/2^{i+j-1}$ hold for any event 
$F$ generated by the events $\vec{uv}\in E(D):(u,v)\neq (x,y)$.

Showing that, with high probability, $D$ has no isolated vertex is relatively straightforward, and may be easily checked by the reader.

We introduce now some new notation. We set $\delta=1-\eta$, $\gamma=(1+\eta)/2$, 
$k=\lceil \log_{2}{\frac{64}\delta}\rceil$, and $K = K_{\eta} = \lceil(32+2k)2^8\delta^{-2}\rceil$. 
 An orientation $F$ of an $n$-cycle has $bias_{\eta}(F)\le n$.  
This example suffices for all $n$ such that $n \le Kn/\log_{2}{n}$,  i.e., $n\le 2^{K}$.   
Fix $n\ge 2^{K}$, and let $D$ be the random oriented graph defined above.  
The proof of Theorem \ref{biasex} is completed by proving that with high probability 
$bias_{\eta}(D)\le Kr$ (recall that $r=\lceil \frac nl\rceil$).

Thus, our task is to bound the probability that there is a pair $A,B\ssq V$ 
with $e(A,B)> Kr$ and $e(B,A)\le \eta e(A,B)$.  
If we consider this event separately for each pair $A,B$, we end up considering far too many events. 
 So we restrict our attention to a certain subset of this family of events. 
 We define a family $\mathcal{R}=\mathcal{R}(\eta)$ 
of pairs $(R,R^{*})$ of subsets of $V$ such that each biased set $E(A,B)$, for $A$ and $B$ subsets 
of $V$, provide two elements $(R,R^*)$ and $(T,T^*)$ of $\mathcal R$ 
which do not belong to four type of events,  concerning pairs of elements of $\mathcal R$.
This will then reduce the problem to bounding the probabilities of these four events; we prove then
these bounds and conclude the proof of Theorem~\ref{biasex}.

To define the family $\mathcal{R}$ below, we first need to associate an (increasing with respect to inclusion)
 function defined on the whole family of subsets of $V$. For a subset $R$ of $V$, and $i=1,\dots , l$, 
let $R_{i}:=R\cap V_{i}$, $r_{i}:=|R_{i}|$, and define $s(R)$ by 
$$s(R):=\sum_{i=1}^{l}\frac{r_{i}}{2^{i}}\, .$$ 
Note that since each $r_i$ is bounded by $r=\lceil \frac nl\rceil$, $s(R)$ itself is strictly smaller than
$r$. Notice also that if $R \subseteq A$, then $s(R)\leq s(A)$. 
For $j = 1,\dots,2l$, consider the collection of sub-intervals $I_j$ of $[0,r)$
 defined by  
$$I_j:=\Bigl[2^{-j-1}r, 2^{-j+1}r\Bigr).$$
By the choice of $l$ (as the largest integer not exceeding $\log_2 n$), 
we infer that for every non-empty subset $R$ of $V$, there is a $j$ such that 
$s(R) \in I_j$. 

 For each $j=1,\dots, 2l$, we denote by $\mathcal{R}_{j}$ the family of pairs of subsets 
$(R,R^*)$ such that $R \subseteq \bigcup_{i\leq j+k} V_i$, $R^* = R \cup \bigcup_{i>j+k} V_i$, 
and $s(R) \in I_{j}$.  Formally,
 $$ \mathcal{R}_{j}:=\Big\{\,\Big (R,R\cup\bigcup_{i>j+k}V_{i}\Big ): \quad R\ssq \bigcup_{i=1}^{j+k}V_{i}\quad\textrm{ and } \quad s(R)\in \Big[2^{-j-1}r, 2^{-j+1}r\Big)\,\Big\}\,.$$  
The family $\mathcal{R}$ is defined as the union of all the sets $\mathcal R_J$, i.e.,
 $$ \mathcal{R}:=\bigcup_{j=1}^{2l}\mathcal{R}_{j}\, .$$

The importance of the family $\mathcal R$ defined above is apparent in the following two lemmas, 
namely that, the pairs $(R,R^*)\in \mathcal R$ "cover`` every non-empty subset of $V$, 
and the size of $\mathcal R$ is sufficiently small.
\begin{lemma}\label{lem:subset}
For every non-empty subset $A \subseteq V$, 
there is an element $(R,R^*) \in \mathcal R$ such that $R \subseteq A \subseteq R^*$
\end{lemma}
\begin{proof}
 Let $1\leq j\leq l$ be such that $s(A) \in [2^{-j}r,2^{-j+1}r)$. Define 
$R = \bigcup_{i\leq j+k} (V_i\cap A)$ and $R^* = R\cup \bigcup_{i>j+k}V_i$. 
Since $\sum_{i\geq j+k+1} \frac{|V_i\cap A|}{2^i} < 2^{-j-k}r$, it follows that $2^{-j-1}r\leq s(R)<2^{-j+1}r$, i.e., 
$(R, R^*) \in \mathcal R_j\subset \mathcal R$, and the claim follows.
\end{proof}

\begin{lemma}\label{lem:size} $|\mathcal{R}|\le \exp\bigl((k+8)r\bigr)$\end{lemma}

\begin{proof} It suffices to prove the bound $|\mathcal{R}_{j}|\le \exp\bigl((k+7)r\bigr)$ 
for all $1\leq j\leq 2l$. 
Fix such a $j$. We have to prove 
that there are at most $\exp\bigl((k+7)r\bigr)$ sets $R\subset \bigcup_{i=1}^{j+k}V_{i}$ 
for which $s(R)\in [2^{-j-1}r,2^{-j+1}r)$.  
In fact we will only use the upper bound $s(R)< 2^{-j+1}r$ to bound the cardinality of $\mathcal R_j$. 

Define $R_i:=R\cap V_i$ so that $r_i= |R_i|$. First remark that the bound $s(R)< 2^{-j+1}r$ 
immediately implies that 
\begin{equation}\label{eq:card}
 r_{i} \le 2^{i+1-j}r \qquad \textrm{for all} \qquad i\le j+k.
\end{equation}
What we shall actually prove is that 
the number of $(j+k)$-tuples $(R_1,\dots,R_{j+k})$ with 
the property that $R_i \subseteq V_i$ and $r_i \leq 2^{i+1-j}r$, is bounded
by $\exp\bigl((k+7)r\bigr)$. For $i=j-1,\dots,j+k$, the inequalities in (\ref{eq:card}) 
are automatically verified, i.e., $R_{j-1}, R_{j}, \dots, R_{j+k}$ 
can be any subset of $V_{j-1},V_{j}, \dots, V_{j+k}$, respectively. 
There are $2^{(k+2)r} < \exp\bigl((k+2)r\bigr)$ ways to choose $R_{j-1},R_{j},\dots,R_{j+k}$.


\noindent The number of ways to choose the subsets $R_1,\dots,R_{j-2}$ can be bounded as follows. 
There are at most $r^{j-2} \le r^l \leq \exp(r)$ ways to choose the sequence of 
set sizes $(r_{1},\dots ,r_{j-2})$, and for each choice of $(r_{1},\dots ,r_{j-2})$, 
the number of sets $(R_{1},\dots , R_{j-2)})$ with corresponding cardinalities 
is bounded by the product of binomial coefficients $\binom{|V_i|}{r_i}$ for $i=1,\dots,j-2$.  
By well known properties of binomial coefficients, and since $r_i \leq r/2$ 
by inequalities (\ref{eq:card}) for $i\leq j-2$, this latter product takes its maximum 
exactly when each $r_{i}$ 
takes its maximum allowed value, i.e., for $r_{i}=2^{i+1-j}r$ for each $1\le i\le j-2$. 
We infer that the number of $(j-2)$-tuples $(R_1,\dots,R_{j-2})$ 
with cardinality constraints above is bounded 
by
\begin{align*}
 \exp(r) \prod_{i=1}^{j-2}\, &\binom{r}{2^{i+1-j}r} \leq \exp(r) \prod_{i=1}^{j-2}\, 
\Bigl( e \frac{r}{2^{i+1-j}r}\Bigr)^{2^{i+1-j}r}\\
&\leq \exp(r)\,\exp\Bigl(\sum_{i=1}^{j-2}2^{j-1-i}r\Bigr)\, 
\exp\Bigl(\sum_{i=1}^{j-2}(i+1-j)2^{j-1-i}r\Bigr)\leq \exp(5r),
\end{align*}
and the lemma follows. Note that in the first inequality above, 
we used the well known bound $\binom{a}{b}\le (e\, a/b)^{b}$ on binomial co-efficients 
(using the fact that $x\log(y/x)\le \sqrt{xy}$ for all positive reals $x,y$).
 \end{proof}

 Lemma~\ref{lem:size} shows that the number of pairs $(R,R^*), (T,T^*) \in \mathcal R$ is bounded by 
$\exp\bigl((2k+16)r\bigr)$. This is the main property we will use in the following. 
Namely, the four type of events that we are going to define below concern 
pairs of elements of $\mathcal R$, 
and have the property  that the probability 
that one of these event happens for a given pair of elements in $\mathcal R$,
is bounded by $\exp\bigl(-(32+2k)r\bigr)$. The bound on the total number of pairs 
then show that with (high) positive probability (as function of $r$) non of 
the four events happen. This will then complete the proof of Theorem~\ref{biasex} 
since as we said, we will show that the event there
exists a biased $E(A,B)$ is contained in one of these four events.   

In the following, we use the notation $\bar{e}(R,T)$ for the number of 
edges between $R$ and $T$ in the underlying graph $G$, so that $\bar{e}(R,T)=e(R,T)+e(T,R)$.  
We also write $\tilde{e}(R,R^{*};T,T^{*})$ for $e(R^{*},T^{*})-e(R,T)$.  
Consider the following four events concerning a given pair of elements $(R,R^*)$ and $(T,T^*)$ 
in $\mathcal R$.

\begin{itemize}
 \item[$\mathbf{F_I}$:] For a given pair of elements $(R,R^{*}),(T,T^{*})\in \mathcal{R}$
 with $s(R)s(T)\ge Kr/4$, 
\begin{center}
  $\bar{e}(R,T)\le s(R)s(T)/2$. 
\end{center}
 \item[$\mathbf{F_{II}}$:] For a given pair of elements $(R,R^{*}),(T,T^{*})\in \mathcal{R}$ with 
$\bar{e}(R,T)\ge Kr/8$,
\begin{center}
  $e(T,R)< \gamma e(R,T)$.
\end{center}
 \item[$\mathbf {F_{III}}$:] For a given pair of elements $(R,R^{*}),(T,T^{*})\in \mathcal{R}$ 
with $s(R)s(T)\ge Kr/4$, 
\begin{center}
 $\tilde{e}(R,R^{*};T,T^{*})\ge \delta s(R)s(T)/4$.
\end{center}
 \item[$\mathbf{F_{IV}}$:] For a given pair of elements $(R,R^{*}),(T,T^{*})\in \mathcal{R}$ 
with $s(R)s(T)\le Kr/4$, 
\begin{center}
 $e(R^{*},T^{*})> Kr$.
\end{center}
\end{itemize}

%

The following lemma shows us that the event of interest to us is contained 
in the union of the above four type of events for pairs of elements of $\mathcal R$, 
 and so combining the bounds on the probabilities of $\mathbf{F_{*}}$ 
with Lemma~\ref{lem:size} will conclude the proof of Theorem \ref{biasex}.

\begin{lemma} The event that there is a pair $A,B\ssq V$ with 
$e(A,B)> Kr$ and $e(B,A)\le \eta e(A,B)$, is contained in the union 
$\mathbf{F_{I}}\cup \mathbf{F_{II}}\cup \mathbf{F_{III}}\cup \mathbf{F_{IV}}$ over all 
pair of elements $(R,R^*), (T,T^*)$ of $\mathcal R$.\end{lemma}

\begin{proof}  Suppose none of the events $\mathbf{F_{I},F_{II},F_{III},F_{IV}}$ 
occur for any pair $(R,R^*), (T,T^*)$ in $\mathcal R$.
 We have to show 
that there is no pair $A,B\ssq V$ with 
$e(A,B)> Kr$ and $e(B,A)\le \eta e(A,B)$. 
For the sake of a contradiction, suppose $A,B\ssq V$ be such a (biased) pair.
 Applying Lemma~\ref{lem:subset}, let $(R,R^*)$ and $(T,T^*)$ be two elements 
of $\mathcal R$ such that $R\ssq A\ssq R^{*}$ and 
$T\ssq B\ssq T^{*}$. We will consider the four events $\mathbf{F_*}$ above 
for the pair $(R,R^*), (T,T^*)$, and the pair $(T,T^*), (R,R^*)$, and derive a contradiction.

\noindent $\bullet$  If $s(R)s(T)\le Kr/4$, then, since $\mathbf{F_{IV}}$ does not occur, 
we should have
 $e(A,B)\le e(R^{*},T^{*})\le Kr$, contradiction.

\noindent $\bullet$ If, on the other hand, $s(R)s(T)\ge Kr/4$, then, 
since none of $\mathbf{F_{I}, F_{II}, F_{III}}$
 occur for the pair $(R,R^*), (T,T^*)$, and the pair $(T,T^*), (R,R^*)$, 
we have that $\bar{e}(R,T) > s(R)s(T)/2$, that $e(T,R)\ge \gamma e(R,T)$, that 
$e(R,T)\ge \gamma e(T,R)$, and that 
$\tilde{e}(R,R^{*};T,T^{*})< \delta s(R)s(T)/4$. 
 
It follows that 
$$e(B,A)\ge e(T,R)\ge \frac{1+\eta}{2}e(R,T)= \frac{1+\eta}{2}\Big( e(R^{*},T^{*})-\tilde{e}(R,R^{*};T,T^{*})\Big)\, .$$

Note that $\bar e(R,T) = e(R,T)+e(T,R) \leq (1+\frac 1\gamma)e(R,T)$. 
We infer that $e(R,T) \geq \frac \gamma{\gamma+1} \bar e(R,T) \geq e(R,T)/2 \geq e(R)e(S)/4$. 
From the bounds $e(R^{*},T^{*})\ge e(R,T)\geq s(R)s(T)/4$ and 
$\tilde{e}(R,R^{*};T,T^{*})\le \delta s(R)s(T)/4$,
 we deduce that 
$$\tilde{e}(R,R^{*};T,T^{*})\le \delta e(R^{*},T^{*})/2,$$
 and so $$e(B,A)\ge \frac{1+\eta}{2}\, \big(1-\frac{\delta}{2}\Big) e(R^{*},T^{*})> \eta e(R^{*},T^{*})\ge \eta e(A,B)\,,$$
yielding again to a contradiction.\end{proof}


\noindent So to complete the proof of Theorem \ref{biasex}, by the choice of $K$, we are only left to prove the following lemmas.

%
%
%

\begin{lemma}[Bounding $\mathbb P(\mathbf {F_I})$]\label{LI} 
For any pair $(R,R^{*}),(T,T^{*})\in \mathcal{R}$ with $s(R)s(T)\ge Kr/4$,
$$\mathbb{P}\,\Bigl(\,\bar{e}(R,T)\le s(R)s(T)/2\,\Bigr)\le \exp(-Kr/32).$$ 
\end{lemma}

\begin{lemma}[Bounding $\mathbb P(\mathbf {F_{II}})$]\label{LII} For any pair 
$(R,R^{*}),(T,T^{*})\in \mathcal{R}$ with $\bar{e}(R,T)\ge Kr/8$,
$$\mathbb{P}\,\Bigl(\,e(T,R)<\gamma e(R,T)\,\Bigr)\le \exp(-Kr\delta^2/128).$$ 
\end{lemma}

\begin{lemma}[Bounding $\mathbb P(\mathbf {F_{III}})$] \label{LIII} For any pair 
$(R,R^{*}),(T,T^{*})\in \mathcal{R}$ with $s(R)s(T)\ge Kr/4$, 
$$\mathbb{P}\,\Bigl(\,\tilde{e}(R,R^{*};T,T^{*})\ge \delta s(R)s(T)/4\,\Bigr)\le \exp\Big(\frac{-\delta^{2}Kr}{256}\Big).$$\end{lemma}

\begin{lemma}[Bounding $\mathbb P(\mathbf {F_{IV}})$]\label{LIV} For any pair
 $(R,R^{*}),(T,T^{*})\in \mathcal{R}$ with $s(R)s(T)\le Kr/4$,
$$\mathbb{P}\,\Bigl(\,e(R,T)> Kr\,\Bigr)\le \exp(-Kr/3).$$ \end{lemma}

\begin{proof}[Proof of Lemma \ref{LI}] 
 Recall that $G$ is the random simple graph underlying $D$.  
The random variable $\bar{e}(R,T)$ may be expressed as 
$\sum_{\{x,y\}:x\in R,y\in T}1_{\{x,y\}\in E(G)}$, a sum of independent $\{0,1\}$-valued 
random variables.  Its mean may be bounded below by 
$$\mathbb{E}(\bar{e}(R,T))\ge \sum_{i=1}^{l}\sum_{j=1}^{l} \frac{r_{i}t_{j}}{2^{i+j}}=s(R)s(T)\ge \frac{Kr}{4}\, .$$
 And so, by Chernoff's inequality, Lemma \ref{Chernoff}, 
we have 
$$\mathbb{P}\,\Bigl(\,\bar{e}(R,T)\le s(R)t(R)/2\,\Bigr)\le \exp(-s(R)s(T)/8)\le \exp(-Kr/32)\, .$$
\end{proof}

\begin{proof}[Proof of Lemma \ref{LII}]
 Let $(R,R^{*}),(T,T^{*})\in\mathcal{R}$ be such that $\bar{e}(R,T)\ge Kr/8$.  
Let $R'=R\setminus T$, let $T' =T\setminus R$, and let $S=R\cap T$.  
We may express $e(R,T)$ as $e(R',T')+e(R',S)+e(S,S)+e(S,T')$, and 
$e(T,R)$ as $e(T',R')+e(T',S)+e(S,S)+e(S,R')$.  
Let us introduce the notation $e'(R,T)=e(R',T')+e(R',S)+e(S,T')=e(R,T)-e(S,S)$, 
and similarly $e'(T,R)$.  Let $\bar{e}'(R,T)=e'(R,T)+e'(T,R)$.  
A few lines of calculation show that the event 
that $e(T,R)\le \gamma e(R,T)$ is exactly the event that 
$(1+\gamma)e'(T,R)\le \gamma \bar{e}'(R,T)-(1-\gamma)e(S,S)$.  
The mean of $e'(T,R)$ is $\mu=\bar{e}'(R,T)/2$, 
and so the above inequality becomes 
$$e'(T,R)\le \mu\Big(1-\frac{1-\gamma}{1+\gamma}-\frac{(1-\gamma)e(S,S)}{\mu}\Big)\, .$$  
Setting $\beta=(1-\gamma)/(1+\gamma)+(1-\gamma)e(S,S)/\mu$, an easy calculation shows that 
$\beta\mu\ge (1-\gamma)\bar{e}(R,T)/(1+\gamma)$, and so, from Chernoff's inequality, we deduce that 
\begin{align*}
\mathbb{P}\,\Bigl(\,e(T,R)\le \gamma e(R,T)\,\Bigr)&\le \exp\Big(-\Big(\frac{1-\gamma}{1+\gamma}\Big)^{2}\frac{\bar{e}(R,T)}{2}\Big)
\le \exp\Big(\frac{-Kr\delta^{2}}{64(1+\gamma)^{2}}\Big)\\
&\le \exp\Big(\frac{-Kr\delta^{2}}{128}\Big)\,.
\end{align*}
\end{proof}

\begin{proof}[Proof of Lemma \ref{LIII}] 
Let $(R,R^{*}),(T,T^{*})\in \mathcal{R}$ be such that $s(R)s(T)\ge Kr/4$.  
The definition of $\mathcal{R}$ implies that 
$s(R^{*}\setminus R)\le \delta s(R)/32$, likewise $s(T^{*}\setminus T)\le \delta s(T)/32$.

\noindent (Note that if $(R,R^*)\in \mathcal R_j$, then $s(R^*\setminus R) \leq 2^{-j-k}r = 2^{-j-1}r \delta/32 \leq s(R)\delta/32$.)  

We let $\tilde{P}$ be the set of pairs with one end point $R^{*}$ and the other in $T^{*}$ 
which do not have one end in $R$ and the other in $T$.  
We may bound the expected value of $\tilde{e}(R,R^{*};T,T^{*})$ by 
$$\sum_{\{u,v\}\in \tilde{P}}\mathbb{P}(1_{\{u,v\}\in E(G)})\le \sum_{i=1}^{l}\sum_{i'=1}^{l}\frac{r_{i}t^{*}_{i'}+r^{*}_{i}t_{i'}+r^{*}_{i}t^{*}_{i'}}{2^{i+i'-1}}\, .$$
The latter sum has value 
$2(s(R)s(T^{*})+s(R^{*})s(T)+s(R^{*})s(T^{*}))\le (\delta/8+\delta^{2}/512)s(R)s(T)$, 
and so $\mathbb{E}(\tilde{e}(R,R^{*};T,T^{*}))\le 3\delta s(R)s(T)/8$.  
Applying Chernoff's inequality, Lemma \ref{Chernoff}, 
with $\lambda=3\delta s(R)s(T)/16$ and $\beta=1/2$ we obtain the bound 
$$\mathbb{P}\,\Big(\,\tilde{e}(R,R^{*};T,T^{*})\ge \frac{\delta s(R)s(T)}{4}\,\Big) 
\le \exp\Big(\frac{-\delta^{2}s(R)s(T)}{64}\Big)\le \exp\Big(\frac{-\delta^{2}Kr}{256}\Big)\, .$$\end{proof}

\begin{proof}[Proof of Lemma \ref{LIV}] 
Let $(R,R^{*}),(T,T^{*})\in \mathcal{R}$ be such that $s(R)s(T)\le Kr/4$.  
An upper bound for $e(R,T)$ is given by $\sum_{\{u,v\}:u\in R,v\in T}1_{\{u,v\}\in E(G)}$.  
This is a sum of independent $\{0,1\}$-valued random variables, and so, 
by Chernoff's inequality (Lemma \ref{Chernoff}), 
it suffices to prove that its expected value is at most $Kr/2$.  
This is easily proved, 
$$\sum_{u\in R,v\in T}\mathbb{P}(1_{\{u,v\}\in E(G)})\le \sum_{i=1}^{l}\sum_{i'=1}^{l}\frac{r_{i}t_{i'}}{2^{i+i'-1}}
=2s(R)s(T)\le \frac{Kr}{2}\, .$$\end{proof}

\subsection{Algorithmic approach}

Proposition \ref{realreg} implies that in every regular oriented graph $D$ there are subsets $A,B\ssq V$ with $e(A,B)\ge n/4$ and $e(B,A)=0$.  We now give a polynomial time algorithm which can find a pair $A,B\ssq V$ with $e(A,B)\ge n/4$ and $e(B,A)=0$.

Our algorithm will begin by building a sequence of sets $A_{1},\dots ,A_{n}$ for which $e^{(2)}(A)$ is not too large.  It then selects a set to be $A$ and from that defines $B$.

\begin{algorithm}[h]
\begin{algorithmic}[1]
\STATE Let $t=1$, let $v$ be an arbitrary vertex and let $A_{1}=\{v\}$.
\WHILE{$t<n$}
\STATE Select a vertex $u\in V\setminus A_{t}$ minimising $e^{(2)}(A_{t},\{u\})$.
\STATE Set $A_{t+1}=A_{t}\cup \{u\}$.
\STATE Increase $t$ by one.
\ENDWHILE
\STATE Let $t=\lfloor n/2d\rfloor$
\STATE Let $A=A_{t}$
\STATE Let $B=\{v:e(\{v\},A)=0$
\RETURN $A,B$
\end{algorithmic}
\end{algorithm}

We now analyse the algorithm.

\begin{lemma}\label{pathtoat} For each $t=1,\dots ,n-1$ the set $A_{t}$ satisfies $e^{(2)}(A_{t})\le d^{2}(t^{2}-1)/n$.\end{lemma}

\begin{proof} This is clearly true when $t=1$.  Suppose the bound fails: let $t$ be minimal such that $e^{(2)}(A_{t+1})> d^{2}((t+1)^{2}-1)/n$. 
 Note that, by the definition of $A_{t+1}$, one has that $$e^{(2)}(A_{t+1})=e^{(2)}(A_{t})+\min_{u\in V\setminus A_{t}}e^{(2)}(A_{t},\{u\})\, .$$ 
 Combining this with the observation that the sum of $e^{(2)}(A_{t},\{u\})$ over vertices $u\in V\setminus A_{t}$ is precisely $e^{(2)}(A_{t},V\setminus A_{t})=2d^{2}t-2e^{(2)}(A_{t})$, we obtain that 
$$ 2d^{2}t-2e^{(2)}(A_{t})> (n-t)\Bigl(e^{(2)}(A_{t+1})-e^{(2)}(A_{t})\Bigr)\, .$$  Straightforward calculation, and the use of the bounds we are assuming on $e^{(2)}(A_{t})$ and $e^{(2)}(A_{t+1})$, yield that $t>n-2$.  Hence there cannot exist $t\in \{1,\dots , n-1\}$ for which the bound fails.  \end{proof}

Let us define a quadratic equation $f(t)=dt-d^{2}(t^{2}-1)/n$.  This quadratic obtains its maximum at $n/2d$.

\begin{proposition} Given a $d$-regular oriented graph $D$.  The above polynomial time algorithm finds subsets $A,B\ssq V$ satisfying $e(A,B)\ge n/4$ and $e(B,A)=0$.\end{proposition}

\begin{proof} Let $t=\lfloor n/2d\rfloor$.  The algorithm outputs $A=A_{t}$ and $B=\{v:e(\{v\},A)=0\}$.  It is immediate from the definition of $B$ that $e(B,A)=0$.  We now prove the bound on $e(A,B)$.  Let us note that each arc from $A$ to $V\setminus B$ is the first arc of a path of length two from $A$ to $A$.  Thus $e(A,V\setminus B)\le e^{(2)}(A)$.  And so $e(A,B)=e(A,V)-e(A,V\setminus B)\ge d|A|-e^{(2)}(A)$.  We now use the fact that $A$ is obtained as $A_{t}$.  This implies that $|A|=t$, and, by Lemma \ref{pathtoat}, that $e^{(2)}(A)\le d^{2}(t^{2}-1)/n$.  Substituting these values we obtain that $e(A,B)\ge f(t)$.  Since $f$ is a quadratic with its maximum at $n/2d$ its value in the range $(n/2d-1,n/2d]$ is always at least $f(n/2d-1)$.  A simple calculation shows that $f(n/2d-1)=n/4$, completing the proof.\end{proof}

\section{Concluding Remarks}\label{ConcRem}

We would very much like to see a proof of Conjecture \ref{evencycleconj}.  By Theorem \ref{fourcycle}, the $k=2$ case of this Conjecture is settled. 
 Theorem \ref{sixcycle} proves that every oriented graph with $bias(D)<\ep e(D)^{2}/n^{2}$ contains an oriented six-cycle.  
This is weaker than the $k=3$ case of the conjecture, which requires that every oriented graph with $bias(D)<\ep e(D)^{3/2}/n$ contains an oriented six-cycle. 
 If this result does hold then it is best possible.  One can prove this by considering a random orientation of the largest (with respect to number of edges) simple graph (on $n$ vertices) 
with girth greater than six, further examples being obtained as blow-ups of these first examples.  The fact that these examples have $bias(D)=O(e(D)^{3/2}/n)$ relies on the fact that there is a constant $c>0$ such that for all $n$, there is a simple graph $G$ on $n$ vertices which has girth greater than six and such that $e(G)\ge cn^{4/3}$.  Such graphs may be constructed using the rank two geometries introduced by Tits \cite{Tits}.

Bondy and Simonovits \cite{BS74} proved that a graph with girth greater than $2k$ has $O(n^{1+1/k})$ edges.  It is widely believed that this result is best possible, i.e., it is believed that for each $k\ge 2$, there is a constant $c_{k}>0$ such that for all $n$, there is a graph on $n$ vertices with girth greater than $2k$ and with $e(G)\ge c_{k}n^{1+1/k}$.  This is known to be true for $k=2,3,5$.  If it is true for all $k$, then by considering blow-ups of random orientations of graphs with girth greater than $2k$, we can find oriented graphs with $bias(D)=O(e(D)^{k/(k-1)}/n^{2/(k-1)})$ which contain no oriented $2k$-cycles.  In this case Conjecture \ref{evencycleconj}, if true, is best possible.

Our final remark on Conjecture \ref{evencycleconj} is that it seems unlikely that a straightforward adaptation of our approach will give a proof.  Our proofs focus on what happens locally to a given vertex.  We examine the arcs from the out-neighbourhood of a vertex to the second out-neighbourhood.  However, there are oriented graphs in which each component contains only $e(D)^{2}/n^{2}$ arcs, and so a condition of the form $bias(D)<\ep e(D)^{k/(k-1)}/n^{2/(k-1)}$ (for $k\ge 3$) gives no information about a particular component.  Thus, any successful approach to Conjecture \ref{evencycleconj} must go beyond the aggregation of certain locally observed inequalities.

Up to this point we have never explicitly put any condition on the underlying graph.  What happens if we do?  This may change the problem considerably.  Certainly the examples we have used up to this point have a certain structure to their underlying graph as well as to the orientation.  One particular case that may be of interest is the case in which the underlying graph is random.  The most general question that arises in this context is the following. 

\begin{question} Let $H$ be an oriented graph and let $p=p(n)$ be some function of $n$ (e.g. $p(n)=n^{1/2}$). Then for which function $b(n)$  do we have the following: 
 with high probability, every orientation $D$ of the Erd\H os-R\' enyi random graph $G(n,p)$ satisfying $bias(D)<b(n)$ contains a copy of $H$?
 \end{question}  
The two functions $b(n)$ and $p(n)$ are certainly correlated. In the extreme case, when $p(n)$ goes to $0$ sufficiently slowly that there are (whp) many copies of $H$ is $G(n,p)$, one may ask, similarly to the dense case, whether a condition as weak as $bias(D)<\ep pn^{2}\sim \epsilon e(G(n,p))$ suffices to ensure that (whp) $D$ contains a copy of $H$? 
Let $\bar H$ be the underlying unoriented graph of $H$ with $n_H$ vertices and $e_H$ edges, and with upper density $\delta_H$: this is by definition the maximum density of a subgraph of $H$, i.e., $\delta_H:= \max_{U \ssq V(H)}\, \frac {e(H[U])}{|U|}$.
By the result of Erd\"os and R\'enyi~\cite{ER60} (for the balanced case) and Bollob\'as~\cite{Bol81} (for the general case), there exists a threshold at $p_H = n^{-\frac 1{\delta_H}}$ for the event that $G(n,p)$ contains a copy of  $H$. Also, large deviation and concentration results for the number of homomorphic copies of $H$ in $G(n,p)$ around $n^{n_H}p^{e_H}$ are known (e.g., see~\cite{Ruc88, JLR90, Vu01, JLR02}). This is roughly the expected number of copies of $H$ in $G(n,p)$ modulo a constant factor imposed by the automorphisms of $H$. It is easy to see that, in order to have a result ensuring the existence of $H$ as an oriented subgraph for an orientation of a random graph $G(n,p)$ with a condition as weak as $bias(D)<\ep pn^{2}$, we should have $n^{n_H}p^{e_H} = \Omega(pn^2)$. So, based on some related questions on quasi-random unoriented graphs, it is reasonable to conjecture that the following should be true.

\begin{conjecture} Let $\bar \delta_H =\max_{U \ssq V(H),\, |U|\geq 3} \frac{e(H[U])-1}{|U|-2}$ and  $p(n)>> n^{-\frac 1{\bar\delta_H}}$. Then with high probability the following holds: an orientation $D$ of $G(n,p)$ has an oriented copy of $H$ as subgraph provided that $bias(D) \leq \epsilon pn^2$.
\end{conjecture}    

Indeed, the above conjecture follows from a conjectural version of the embedding lemma in the sparse case (c.f., conjecture 8 in~\cite{Koh97} and conjecture 3.11 of~\cite{GS05}). This can be seen by mimicking the proof given in~\cite{GS05} of a related conjecture on the values of $ex(G(n,p),H)$ in the case $pn^{\frac 1{\delta_H}} \rightarrow \infty$  assuming the sparse version of the embedding lemma (c.f., Conjecture 3.15  and Theorem 3.16 in~\cite{GS05}).
 
 \vspace{.3cm}

We now discuss other parameters that may be of interest.  Can one ensure the presence of an oriented four-cycle by putting an upper bound on $\max\{e(A,B)-e(B,A):A,B\ssq V\}$?  The answer is of course ``Yes'', as it follows from Theorem \ref{fourcycle} that an upper bound of the form $\ep e(D)^{2}/n^{2}$ suffices.  Surely this is not best possible.  What is the largest upper bound that suffices?  Another parameter that one might consider is the size of the largest eigenvalue.  What is the largest upper bound on the largest eigenvalue that can ensure the presence of an oriented four-cycle?  Finally, another parameter that may be of interest is the size of the largest one-way subgraph (c.f. Section~\ref{regran}):  let $ow(D)=\max\{e(A,B):A,B\ssq V\, ,\, e(B,A)=0\, \}$.  Can an upper bound on $ow(D)$ ensure the presence of an oriented four-cycle in $D$?  It is quite possible that a bound of the order of $e(D)^{2}/n^{2}$ suffices. 

\begin{conjecture} There exists a constant $\ep>0$ such that every oriented graph $D$ with $ow(D)<\ep e(D)^{2}/n^{2}$ contains an oriented four-cycle.\end{conjecture}

\vspace{.4cm}

\paragraph{Acknowledgment} The starting point for the research presented in this paper was a problem related to the Cacceta-H\"aggkvist conjecture which was asked by Pierre Charbit, and was communicated to the authors by St\'ephan Thomass\'e. 
The authors thank both for fruitful discussions. Part of this research was done while O. A. was PhD student at \'Ecole Polytechnique and Mascotte Project, S. G. was PhD student at Cambridge University, and F. H. was PhD 
student in Mascotte joint project at CNRS-UNSA-INRIA. The authors thank all these institutions for providing the possibility for several visits and for their hospitality.

\newpage

\appendix
\section{Dense Case Using Szemer\'edi's Regularity Lemma}\label{sec:appendix}
In this appendix, we prove the following theorem. 
\begin{theorem}\label{conjdense}
Let $h\in \mathbb N$ be a fixed integer and $H$ be an oriented graph on $h$ vertices. Then for every
$\nu \in (0,1)$, there exists real numbers $\epsilon>0$, $0 < \beta =\beta(H)< \frac{1}{2}$, and an
integer $N$ such that every oriented graph $D$ on $n\geq N$ vertices, with $\beta n^2$
arcs with $bias_\nu(D) \leq \epsilon e$, contains $H$ as a subgraph.

When $H$ is an orientation of a bipartite graph, $\beta$ can be any arbitrary small positive value. 
\end{theorem}

The proof of Theorem \ref{conjdense} is based on the use of Szemer\'edi's regularity lemma, so we need first to state it.
Let us introduce the following notation:
Let $U$ and $W$ be two disjoint subsets of the vertex set $V$ of a given (oriented) graph $D$. The number of
edges between $U$ and $W$ without taking into account the orientation, if any, is denoted by $\overline{e}(U,W)$. Recall that the {\it density} of the pair $U,W$ is
$$d(U,W)\:=\: \frac{\overline{e}(U,W)}{|U||W|}\:.$$

\begin{definition}[Uniform Pairs]\label{definition:uniform}
Let $\sigma \in (0,1)$ be a real number. The pair $(U,W)$ is called $\sigma$-{\it uniform} if, for all subsets $U' \subset U$ and $W' \subset W$ with $|U'|>\sigma|U|$ and $|W'|>\sigma|W|$, we have $|d(U',W')-d(U,W)| < \sigma$.
\end{definition}

A simple consequence of the $\sigma$-uniformity condition above is that bipartite subgraphs induced by uniform pairs must be roughly regular. More precisely we have the following lemma (the proof of which can be found for example in~\cite{Die05}).
\begin{lemma}\label{regular}
Let $(U,W)$ be a $\sigma$-uniform pair and let $d(U,W) = d$. Then we have
$$|\{u \in U : |\Gamma(u) \cap W| > (d-\sigma)|W|\}| \geq (1-\sigma)|U|$$
and similarly
$$|\{u \in U : |\Gamma(u) \cap W| < (d+\sigma)|W|\}| \geq (1-\sigma)|U|.$$
\end{lemma}

\begin{definition}[Equipartitions]
An {\it equipartition} of $V(G)$ into $m$ parts is a partition $V_1, \dots, V_m$ such that
$\lfloor \frac n m \rfloor \leq |V_i| \leq \lceil \frac n m \rceil$, for $1\leq i \leq m$, where $n = |V(G)|$. The partition is $\sigma$-uniform if
$(V_i, V_j)$ is $\sigma$-uniform for all but $\sigma {m \choose 2}$ pairs $1 \leq i < j \leq m$.
\end{definition}

\begin{lemma}[Szemer\'edi Regularity Lemma~\cite{sze75}]
Let $\sigma$ be a real number in $(0, 1)$ and let $l$ and $N$ be two natural numbers. Then there exists $L=L(l,\sigma)$ such
that every graph with at least $N$ vertices has a $\sigma$-uniform equipartition into $m$ parts for some $l \leq m \leq L$.
\end{lemma}

We are now ready to prove the following directed version of the embedding lemma in the presence of an appropriate condition on the bias.

\begin{lemma}\label{subgraphH}
Let $h$ and $m\in \mathbb N$ be two fixed integers with $m\geq h$, $H$ be an oriented graph on $h$ vertices, and $\nu$ a real number in  $(0,1)$. Let $D$ be a digraph with $e$ arcs, containing disjoint subsets $V_1,\dots,V_h$ of size $u$ or $u+1$ of the vertex set $V$ such that 
\begin{itemize}
\item[(i)] $u \geq \frac nm$;
\item[(ii)] For each arc $(v_i,v_j) \in E(H)$, $(V_i,V_j)$ is $\sigma$-uniform;
\item[(iii)] For each $i\neq j$ with $(i,j)\in E(H)$, $d(V_i,V_j)\geq \eta$ for some  positive real number $\eta>0$;
\item[(iv)] For each $i\neq j$ with $(i,j) \in E(H)$, $bias_\nu(D(V_i \cup V_j)) \leq \epsilon e$, where $\epsilon=\frac{\sigma^2(\frac \eta{1+\frac 1{\nu}}-\sigma)}{m^2}$.
\end{itemize}
 Furthermore, suppose that
$\bigl(\frac {\eta}{1+\frac 1\nu}-\sigma\bigr)^h\geq(h+1)\:\sigma$,  and finally $ \sigma u\geq 1$.
Then $D$ contains an oriented subgraph isomorphic to $H$.
\end{lemma}

\begin{proof}
Let $\{v_1,\dots,v_h\}$ be the set of vertices of $H$.
We prove by induction on $h$ that the following property holds for all $l=0,\dots,h$:

\hspace{1cm}

\noindent {\it Property} $\mathcal{P}(l)$: There exists a sequence $x_1,\dots, x_l$ so that 
\begin{itemize}
\item[I.] for all $i=1,\dots,l$, $x_i \in V_i$; and 
\item[II.] for every $j$, $l < j \leq h$, there is a set $X^l_j \subset V_j$ of potential candidates for $x_j$ at stage $l$ of size larger than $\bigl(\frac \eta{1+\frac 1\gamma}-\sigma)^{N(j,l)}u$. More precisely, if we set $N^-(j,l):=\{\:x_i \:|\: 1 \leq i \leq l\: \textrm{and}\ v_iv_j \in E(H)\:\}$, $N^+(j,l):=\{\:x_i \:|\: 1 \leq i \leq l\: \textrm{and}\ v_jv_i \in E(H)\:\}$, and $N(j,l)=N^-(j,l)\cup N^+(j,l)$, then for any $y_j \in X^l_j$ and for every $x_i \in N(j,l)$, $x_iy_j \in E(D)$ or $y_jx_i\in E(D)$ provided that $x_i\in N^-(j,l)$ or $x_i\in N^+(j,l)$ respectively, and  moreover, $|X^l_j| \geq \bigl(\frac \eta{1+\frac 1\gamma}-\sigma)^{N(j,l)}u$.
\end{itemize}
\vspace{.2cm}
\noindent It is clear that the property $\mathcal{P}(0)$ holds, for this we just take $X^0_j = V_j$.

\noindent Suppose now that $\mathcal{P}(l)$ is true, we will prove that $\mathcal{P}(l+1)$ also holds.

Let $T^+$ be the set of indices $j$ for which we have an outgoing arc $v_{l+1}v_{j}$ from $v_{l+1}$ to $v_j$, i.e., 
$$T^+ = \{\:j > l + 1\:|\: v_{l+1}v_{j} \in E(H)\:\}.$$ 
Similarly, we define $T^-$ for the set of incoming arcs:
$$T^- = \{\:j > l + 1\:|\: v_{j}v_{l+1} \in E(H)\:\}.$$

 We should select $x_{l+1} \in X^l_{l+1}$ in such a way to have arcs from 
$x_{l+1}$ to the vertices of the set $X_j^{l+1}$ in the appropriate direction, and such that updating $X^l_j$ to $X^{l+1}_{j}$ we respect the condition $II$ of the property $P(l+1)$. In the following, let $Y_j$ be the
set of vertices of $X^l_{l+1}$ which does not have enough arcs to $X_j^l$ and
hence are not good candidates for $x_{l+1}$. 
More precisely,  for each $t \in T^+$ (resp. for each $t\in T^-$), we define the set $Y_t$ as follows 
$$ Y_t\::=\: \{\:y \in X^l_{l+1}\: |\:\:\: |\Gamma^+(y) \cap X^l_t| < \bigl(\frac \eta{1+\frac 1\nu}-\sigma\bigr)|X^l_t|\:\}$$
$$ (\textrm{resp.}\:Y_t\::=\: \{\:y \in X^l_{l+1}\: |\:\:\: |\Gamma^-(y) \cap X^l_t| < \bigl(\frac \eta{1+\frac 1\nu}-\sigma\bigr)|X^l_t|\:\}\:)\:.$$

By property $\mathcal{P}(l)$ and the hypothesis, we have: 
\begin{align*}
|X^l_t| & \geq (\frac {\eta}{1+\frac 1\nu}-\sigma)^{N(j,l)}|V_{t}| \geq (\frac {\eta}{1+\frac 1\nu}-\sigma)^h |V_{t}| \\
&\geq (h+1)\sigma |V_{t}|> \sigma |V_{t}| >\sigma u\:.
\end{align*}

It follows from Definition~\ref{definition:uniform}, the fact that $bias_\nu(X^l_t \cup Y_t)\leq \epsilon e$, and the choice of parameters that for all $t\in T^+\cup T^-$, $|Y^t|\leq \sigma |V_{l+1}|$. To see this, suppose for the sake of a contradiction that this is not the case, 
i.e., $|Y_t| > \sigma |V_{l+1}|$. By the choice of $\epsilon$ and $|V_{l+1}| \geq \frac n{m}$, this means that $|Y_t||X^l_t| > \sigma^2u^2\geq \frac {\sigma^2}{m^2}n^2\geq \frac 1{(\frac{\eta}{1+\frac 1\nu}-\sigma)}\epsilon e$. This means that
$\epsilon e < (\frac \eta{1+\frac 1\nu}-\sigma) |X_t^l||Y_t|.$

From now on without loss of generality we will suppose that $t \in T^+$ (the other case follows similarly). Two cases can happen:
\begin{itemize}
\item if $e(X^l_t,Y_t)<\epsilon e$, then we have $e(X^l_t,Y_t) <  (\frac{\eta}{1+\frac 1\nu}-\sigma)|X^l_t||Y_t| < \frac 1\nu (\frac{\eta}{1+\frac 1\nu}-\sigma)|X^l_t||Y_t|$. 
By the definition of $Y_t$, we also have $e(Y_t,X^l_t) < (\frac \eta{1+\frac 1\nu}-\sigma)|X^l_t||Y_t|.$ It follows that 
$d(X^l_t,Y_t) = \frac {\bar e(X_t^l,Y_t)}{|X_t^l||Y_t|} < \eta -\sigma.$  We saw that $X^t_l > \sigma |V_t|$ and by our assumption we have $Y_t > \sigma |V_{l+1}|$. These all will provide a contradiction to the $\sigma$-regularity, since $d(V_t,V_{l+1})\geq \eta.$  
\item if $e(X^l_t,Y_t) \geq \epsilon e$ then by the condition on the bias,
$e(Y_t,X^l_t) \geq \nu e(X^l_t,Y_t)$, and this means that $e(X^l_t,Y_t)\leq \frac 1{\nu}e(Y_t,X^l_t)< \frac 1{\nu}(\frac \eta{1+\frac 1\nu}-\sigma)|X_t^l||Y_t|$. 
So 
$$\overline{e}(X^l_t,Y_t)=e(X^l_t,Y_t)+ e(X^l_t,Y_t)\leq (1+1/\nu)(\eta/(1+1/\nu)-\sigma)|X^l_t||Y_t|.$$ 
This implies that $d(X^l_t,Y_t)\leq (1+1/\nu) (\eta/(1+1/\gamma)-\sigma) < (\eta-\sigma)$ and this be again a contradiction as in the first case.
\end{itemize}

We infer that $Y_t\leq \sigma |V_{l+1}|$.
Therefore $|X^l_{l+1}\setminus (\cup_{t\in T^-\cup T^+}Y^t)|\geq
(\eta/(1+1/\nu)-\sigma)^{N(l+1,l)}|V_{l+1}| -|T^-\cup T^+|\sigma |V_{l+1}|\geq ((\eta/(1+1/\nu)-\sigma)^{N(l+1,l)}-h \sigma) |V_{l+1}|\geq \sigma u\geq 1$, and so the set $X^l_{l+1}\setminus (\cup_{t\in T^-\cup T^+}Y^t)$ is not empty. This means that we can select $x_{l+1} \in X^l_{l+1}\setminus (\cup_{t\in T^-\cup T^+}Y^t)$ in such a way to respect the condition I of $\mathcal P(l+1)$. By the choice of $x_{l+1}$, it is easy to verify that by taking $X^{l+1}_t = X^l_t \cap \Gamma(x_{l+1})$ for $t\in T^-\cup T^+$ and $X^{l+1}_t=X^l_t$ for $t \notin T^-\cup T^+$, we respect also the condition II of $\mathcal P(l+1)$. And so $\mathcal{P}(l+1)$ is true.
The above induction shows that $\mathcal P(l)$ holds for all $l, 0\leq l \leq h$, so $x_1, \dots, x_h$ can
be found as desired. These vertices form a copy of $H$.
\end{proof}

\begin{proof}[Proof of Theorem \ref{conjdense}]
We consider an oriented graph $H$ on $h$ vertices, a real number $\nu \in (0,1)$, a sufficiently large integer number $n$ ($n\geq N$, for $N$ being fixed later),  a digraph $D$ on $n$ vertices with $e=\beta n^2$ ($\beta$ to be fixed later) arcs whose bias is at most $\nu \epsilon e$ ($\epsilon$ to be fixed later).  In what follows, for an oriented graph $D$, we denote by $\overline D$ the corresponding non-oriented graph obtained by forgetting the orientation of $D$.

We now apply Szemer\'edi's regularity lemma to $\overline{D}$ for $l$ and some
$\sigma$ to be fixed later. We set $l$ sufficiently large but fixed such that a graph $G$ on at least $l$ vertices with $(1-\sigma)\bigl (^l_2\bigr)$ edges contains a copy of $\overline H$ as subgraph.
We obtain a $\sigma$-uniform equipartition of $V(G)$ into $m$ parts $V_1,\dots,V_m$
with $l\leq m \leq L$. 
We require $\beta$ to satisfy: 
$$\beta n^2 \geq \eta \frac {n^2}2 +\frac {n^2}m+ ex(m,\overline H) (\frac nm)^2,$$ 
for some small enough $\eta$.
When $H$ is non bipartite, $ex(m,h) \geq (1-\sigma)\bigl(^m_2\bigr)$, and it gives a lower bound on $\beta$ in terms of $\eta$ and $\sigma$ (as we will see $\eta$ and $\sigma$ are correlated in such a way that $\eta$ small enough will imply $\sigma$ small enough and $\sigma <\eta$). When $H$ is bipartite, we can choose $\beta$ as small as we want as long as $\eta$ is also chosen small (it just implies $n$ to be big enough).

The idea is that from this partition we construct a graph $G$, in which each vertex represents a  part $V_i$ of the corresponding partition. Two vertices of
$G$ are linked together if the corresponding parts have a density of at least
$\eta$ in between them (i.e., $(v_i,v_j) \in E(G)$ iff $d(V_i,V_j)\geq \eta$) for some $0<\eta<1$ which will be chosen according to the hypothesis of Lemma~\ref{subgraphH}. As long as $\eta$ is not too
big, there is some $\beta$ that ensures us that $e(G) \geq ex(m,\overline{H})$.
Consequently we get a copy of $\overline{H}$ in $G$, wlog it uses the vertices
$v_1, \dots,v_h$, $v_i\in V_i$, and we add $m-h$ vertices $v_{h+i}\in V_{h+i}$.

We will apply Lemma \ref{subgraphH} with appropriate values of $\epsilon$, $\eta$ and $\sigma$. We fix $l$ large enough and  we choose $\eta$ small enough such that a graph $G$ with $(1-\eta^h)\bigl(^m_2)$ edges contains a copy of $\overline H$ for every $m\geq l$. Once $l$ and $\eta$ are fixed, we choose $\sigma$ such that
$\bigl(\frac \eta{1+\frac 1\nu}-\sigma\bigr)^h\geq (h+1) \sigma$. Remark that such a positive $\sigma$ always exists and $\sigma<\eta$. To see this, remark that the continuous function $f(x)=\bigl(\frac \eta{1+\frac 1\nu}-x\bigr)^h-(h+1) x$ has a positive value at $x=0$ and so for some small $x$ close to zero it will still remain positive. We apply Szemer\'edi's regularity lemma to $\sigma$ and $l$ and $\overline D$ on $n$ vertices with $\beta n^2$ edges ($\beta$ chosen as above with respect to this choice of $\sigma$ and $\eta$) for $n$ large enough. There exists $L$ such that a partition to $m$ parts exists with $l\leq m\leq L$. We choose $N$ large enough such that $\sigma \frac NL \geq 1$. Remark that $u\geq \frac NL$ in Lemma~\ref{subgraphH} so that $\sigma u \geq 1$. Let us choose $\epsilon = \frac {\sigma^2(\frac \eta{1+\frac 1\nu}-\sigma)}{m^2}$. We claim that for these choices of parameters every oriented graph $D$ on $n\geq N$ vertices with $bias(D) \leq \nu \epsilon e$ contains a copy of $H$. Let $V_1,\dots,V_h$ be the parts of the equipartition which provide a copy of $\overline H$. We apply lemma~\ref{subgraphH} to this data. We should verify that the conditions of lemma are verified:
\begin{itemize}
\item[(i)] We have easily $u \geq \frac nm$;
\item[(ii)\&(iii)] For each arc $(i,j) \in H$, the pair $(V_i,V_j)$ is $\sigma$-uniform and $d(V_i,V_j) \geq \eta$. This is true because we find the copy of $H$ in the graph $G$ formed by the pairs which were $\sigma$-uniform and have the required density condition;
\item[(iv)] For each arc $(i,j)$, $bias_\nu(D(V_i,V_j)) \leq \epsilon e$, where $\epsilon = \frac {\sigma^2(\frac \eta{1+\frac 1{\nu}}-\sigma)}{m^2}$. This is true because $D$ has this property, and so $D(V_i,V_j)$ inherits this property from $D$.
\end{itemize} 
Furthermore, the conditions $\bigl(\frac \eta{1+\frac 1\nu}-\sigma\bigr)^h\geq(h+1)\sigma$ and $\sigma u \geq 1$ are also verified. So we can apply Lemma~\ref{subgraphH} to find a copy of $H$ in $D$.
\end{proof}

A result from N. Alon and A. Shapira \cite{AS04} says that if we have an oriented graph such that
we have to remove $\epsilon n^2$ arcs before it becomes $H$-free, then it has
many copies of $H$. The exact statement of their theorem is the following:

\begin{theorem}[Alon and Shapira~\cite{AS04}]\label{Alon}
For every fixed $\ep$ and $h$, there is a constant $c(\ep,h)$ with the following
property: for every fixed digraph $H$ of size $h$, and for every digraph $D$ of
a large enough size $n$, from which we need to delete $\ep n^2$ arcs to make it
$H$-free, $D$ contains at least $c(\ep,h) n^h$ copies of $H$.
\end{theorem}

From this theorem and Theorem \ref{conjdense}, we may deduce the following
Corollary:

\begin{corollary}\label{numberdense}
For every $h$, every oriented graph $H$ on $h$ vertices, there exist
$\epsilon>0$, $0 < \beta < \frac{1}{2}$, $c(\ep,h,\beta)$ and an integer $N$ such that
every digraph $D$ on $n\geq N$ vertices, with more than $\beta n^2$ arcs and with $bias(D) \leq \epsilon e$ contains $c(\ep,h) n^h$ copies of $H$.
\end{corollary}
\begin{proof}(Sketch of proof)
Let $h$ be an integer and $H$ an oriented graph on $h$ vertices, Theorem
\ref{conjdense} gives $\ep$, $\beta$ and $N$ such that every digraph $D$ on
$n\geq N$ vertices, with $\beta n^2$ arcs and with $bias_{\,0.4}(D) \leq  \epsilon
e$ contains $H$ as a subgraph.

Setting $\ep'=\ep\beta/10$, every digraph $D$ on $n\geq N$ vertices, with
$(\beta+\ep')n^2$ arcs and with $bias(D) =bias_{\,0.5}(D)\leq  \epsilon
e$ contains $H$ as a subgraph. 
On top of this, if we delete $\ep' n^2 $ arcs, the remaining oriented digraph $D'$ still contains $H$ as a
subgraph. Indeed $D'$ has more than $\beta n^2$ arcs, and one can easily check that $bias_{\,0.4}(D') \leq \epsilon e$, so Theorem~\ref{conjdense} ensures the existence of a copy of $H$.
  
  We may now apply Theorem \ref{Alon} which will imply the existence of $c(\ep',h) n^h$ copies of $H$ in $D$. Notice that if $\beta+\ep'\geq 1/2$, we can change the value of $\ep'$ to some smaller value in order to  have a number bellow $1/2$, and the reasoning remains unchanged.
\end{proof}

In other words, the above corollary says that given $h$, there exist $\epsilon>0$, $0 < \beta < \frac{1}{2}$, $c(\ep,h,\beta)$ and an integer $N$ such that every digraph $D$ on $n\geq N$ vertices, with $\beta n^2$ arcs and with $bias(D) \leq \epsilon e$ contains the correct order of any orientation of any graph on $h$ vertices.

\end{document}